\documentclass[12pt,a4paper]{article}
\usepackage{latexsym,amsmath,amsthm,amssymb,fullpage}
\usepackage[latin1]{inputenc}
\newcommand{\dom}{\ensuremath{\mathcal{D}}} %domain
\newcommand{\dE}{\ensuremath{\mathbb{E}}} %esperance
\newcommand{\dR}{\ensuremath{\mathbb{R}}} %reels
\newcommand{\R}{\dR}

\newcommand{\var}{\mathbf{Var}} %variance
\newtheorem{theorem}{Theorem}
\newtheorem{proposition}[theorem]{Proposition}
\newtheorem{lemma}[theorem]{Lemma}
\newtheorem{corollary}[theorem]{Corollary}
\newtheorem{fact}[theorem]{Fact}
\newtheorem{setting}[theorem]{Setting}

\theoremstyle{definition}

\theoremstyle{remark} \newtheorem{remark}{Remark}

\def\id{\textrm{Id}}
\def\dmu{\, d\mu}
\newcommand{\poinc}[1]{\ensuremath{{\rm c_p}(#1)}}

\def\O{\mathcal O}
\def\fix{\mathrm{Fix}}

\begin{document}   %----------------------------------------------------
%%%---------------------------------------------------------------------
\title{Invariances in variance estimates}
\author{F. Barthe and D. Cordero\,-Erausquin}
\date{June 12, 2011}

\maketitle

\begin{abstract}
We provide variants and improvements of the Brascamp-Lieb variance inequality which take into 
account the invariance properties of the underlying measure. This is applied to spectral gap 
estimates for log-concave measures with many symmetries and to 
non-interacting conservative spin systems.
\end{abstract}

\section{Introduction}
Poincar\'e type inequalities, which provide upper estimates of variances of functions by energy terms involving quadratic forms in their 
gradients, are versatile tools of mathematical analysis. They allow for example to quantify the concentration of measure or the ergodic behaviour
of evolution processes. 
In the recent years, it has become clear that they %variance and Poincar\'e type inequalities
 provide crucial information on the distribution of mass and on the central limit theorem for convex bodies (see e.g.~\cite{ABP, KLS, Bobkov:1999vk, Guedon:2011wl}). 
Recall that for a Borel probability measure $\nu$ on a Euclidean space $(E, |\cdot |)$, its Poincar\'e constant $\poinc{\nu} \in (0,+\infty]$ is the best constant  for which we have
\begin{equation}\label{defpoinc}
\var_{\nu}(g) \le \poinc{\nu} \int |\nabla g |^2 \, d\nu
\end{equation}
for every  $g\in L^2(\nu)$ locally Lipschitz, where the variance with respect to $\nu$ is defined by
$$\var_{\nu}(g) := \int \Big(g- \int g \, d\nu\Big)^2 \, d\nu.$$
For a random vector $Y\in E$, if $\nu$ is the law of $Y$ (a relation denoted by $Y\sim \nu$), we define  $\poinc{Y}:=\poinc{\nu}$. Let us also recall here that when $\nu$ is log-concave (see below), this Poincar\'e constant is finite, and we have, for instance, the following bound  proved by Kannan-Lovasz-Simonovits~\cite{KLS} and Bobkov~\cite{Bobkov:1999vk}
\begin{equation}
\label{KLS}
\poinc{\nu} \le 4\inf_{a\in E}\int_E |y-a|^2d\nu(y)
\end{equation}
Spectral estimates enter in the asymptotic geometry of convex bodies and log-concave measures mainly  when the measure $\nu$ is \emph{isotropic}. 
%Special emphasis will be given to the case of the function $g(x)=|x|^2$ because of its importance in the asymptotic theory of convex bodies, i
A probability measure $\nu$ (or a random vector $X\sim \nu$)  is isotropic if
$\int_E x \, d\nu(x) = \mathbb E X = 0$ and 
% \quad \textrm{and}\quad 
$\textrm{Cov}(\nu) := \int_E x\otimes x \, d\nu(x)  = \mathbb E X \otimes X  = \id_E$. The KLS conjecture claims that the Poincar\'e constant of isotropic log-concave distributions is universally bounded, i.e. $\sup\{\poinc{\mu}\; ; \ n\ge 1, \mu \textrm{ isotropic log-concave probability on $\R^n$}\} < +\infty$. A weaker conjecture known as the variance conjecture claims a similar result but when we use only the function $f(x)=|x|^2$; since $\mathbb E |\nabla f (X)|^2 = 2 n$ when $X$ is isotropic, this conjecture amounts to the bound 
\begin{equation}
\label{varconj}
\var(|X|^2)\le C \,n
\end{equation}
 for some universal constant $C>0$ and for every isotropic log-concave vector $X\in \R^n$. We shall detail a bit more on this later.

The present paper has been very much inspired by the  work of Klartag~\cite{KlartagPTRF}, who obtained an optimal variance estimate in the case of log-concave 
measures which are invariant under coordinate hyperplane reflections. His approach was based on a new twist in the $L^2$ method, introduced by H\"ormander.
Our goal is to extend these techniques to a more general setting (non log-concave measures, more general symmetries) that is adapted to applications
in other topics. One of our main results, Theorem~\ref{theo:invariance1}  appears as an improvement of the classical Brascamp-Lieb variance inequality,
where the log-concavity assumption may be relaxed in the presence of symmetries and spectral gaps for restrictions to certain subspaces.

\bigskip

The paper is organized as follows. In the next section, we recall the principle of  the $L^2$ method (in the presence of invariances) and present a streamlined version
of  Klartag's argument as an introduction to our further study. Then, in Section \S\ref{vargeneral} we establish an abstract Poincar\'e-type inequality for measures and functions having well balanced invariances.  Our results, which involve the spectral gaps of  conditioned measures orthogonal to fixed-points subspaces of isometries leaving the measure invariant,  go beyond the class of log-concave measures. In  Section~\S\ref{varnorm} we apply these results to estimate the variance of the norm of
 log-concave random vectors with many invariances (as a consequence, we confirm the variance conjecture for  log-concave measures having the invariances of the simplex).
  Section~\S\ref{sec:symmetrisation} shows how to extend the variance estimates of Section \S\ref{vargeneral} to general functions. It uses a symmetrization
  procedure that relies on spectral properties of the group of isometries of the underlying measure.
    Section~\S\ref{sec:spectralgap} is devoted to spectral gap estimates for log-concave probability measures, with, as before, special emphasis on the measure having several invariances by reflections and on the Schatten classes.
  Section~\S\ref{sec:spin} gives an application to the study of spectral gap of conditioned spin systems. The last Section~\S\ref{sec:isotropy} is mainly independent of the rest, as far as methods are concerned. It discusses the isotropy constant of convex bodies having well balanced invariances. In a final Appendix, we collect some useful observations concerning groups generated by reflections which provide examples satisfying our assumptions throughout the paper. 

\bigskip

We conclude this introduction with precisions about the setting of our study and the notation.
We shall be working with a Borel probability measure $\mu$ on $\R^n$ with density,  
$$d\mu(x) = e^{-\Phi(x)}\, dx.$$
When $\Phi$ is convex, we say that $\mu$ is log-concave.
In our proofs, we shall impose for  simplicity the condition~\eqref{defmu} that $\Phi$ is of class $C^2$ on $\R^n$. However, the inequalities we obtain are of course valid for a larger class of measure, by standard approximations that we leave to the reader. For instance,  conditions of the form  $D^2\Phi \ge \rho \id$  can  be interpreted in the sense of distributions. 
More importantly, we shall explain below why the results extend to the case $\mu$ has some convex support $K$ (not necessarily $\R^n$, thus) provided $\Phi$ is $C^2$ on $K$ (we take the value $+\infty$ outside $K$). For instance, all the inequalities proved in the  paper for log-concave measures of the form~\eqref{defmu} remain valid for general log-concave measures having a convex support. In particular, the results apply to $\mu_K$, the normalized Lebesgue measure restricted to a convex body of $K\subset \R^n$, i.e.
$$\mu_K(A) = \frac{|A\cap K |}{|K|}, \qquad \forall A\subset \R^n.$$

 Given an isometry of the standard Euclidean space $(\R^n, \cdot, |\cdot|)$, $R\in \mathcal O_n$ and a function $g$ on $\R^n$,
  we say that $g$ is $R$-invariant if $g\circ R = g$. A measure $\mu$ is $R$-invariant if its push-forward by $R$-invariant is $\mu$, or equivalently if its density is 
  $R$-invariant.  Accordingly, a random vector $X\in \R^n$ is $R$-invariant if $RX$ and $X$ have the same law. These notions extend in the obvious way to $G$-invariance where $G\subset\mathcal O_n$ is a group of isometries. It is worth noting that if a function $g$ (or a measure $\mu$) is invariant under $k$ isometries $R_1, \ldots,R_k \in \mathcal O_n$, then it is invariant under $\langle\{R_i\}_{i\le k}\rangle$, the group of isometries generated by these isometries. For a measure $\mu$ (resp. a random vector $X\in \R^n$), we denote by $\mathcal O_n(\mu)$  (resp. $\mathcal O_n(X)$) the group of all isometries leaving $\mu$ (resp. $X$) invariant. 
%Accordingly, for a random vector $X$, an isometry $R$ belongs to $O_n(X)$ if $RX$ and $X$ have the same law. 
Similarly, for a convex body $K\subset \R^n$, we denote $\mathcal O_n(K):=\{R\in \mathcal O_n \; ;\ RK=K\}$.
Recall that each isometry $R\in \mathcal O_n$ comes with a linear subspace of fixed points, 
$$\fix(R) := \{x \in \R^n \; ; \ Rx = x\}.$$
For a group $G$ of isometries, $\fix(G)$ denotes the intersection of $\fix(R)$ when $R$ runs over $G$.

We shall give special attention to the case of (orthogonal) hyperplane symmetries, which we refer to as~\emph{reflections}. A reflection is characterized by the fact that its fixed points  form an hyperplane $H=u^\perp$ with $u$ a unit vector, and we shall later use the following notation: 
\begin{equation}
\label{defSH}
S_H(x) := S_{u^\perp}(x) := x- 2(x\cdot u) u, \qquad \forall x \in \R^n .
\end{equation} 
For a measure $\mu$, we denote by $\mathcal R_n(\mu)\subset \mathcal O_n(\mu)$ the group generated by the reflections leaving $\mu$ invariant.

\section{Generalities on the $L^2$ method}
\label{sec:L2}

H\"ormander developed the $L^2$ method for solving the $\overline{\partial}$ equation, emphasizing the central role  played by the convexity, or rather plusubharmonicity, of the domain or of the potential $\phi$ (see~\cite{hormander}). 
%He  emphasized the central role  played by the convexity (or rather plusubharmonicity) of the domain and of the potential $\phi$.  Of %course, we are not interested here in solving a gradient equation, and what 
We will use only the  ``easy" part of the method, namely the \emph{a priori} spectral type inequalities (which are often referred to, in the real case, as Poincar\'e or Brascamp-Lieb inequalities). Since H\"ormander's seminal work, the $L^2$ method has been recognised as a powerful way of obtaining spectral inequalities, in particular in the context of statistical mechanics, where the method is also refered to as ``Bochner's method", since in many cases the argument boils down to  Bochner's integration by parts formula~\eqref{bochner}. More recently, links with convex geometry have been emphasized. For instance, the $L^2$ method was used to provide a local proof of functional (complex or real) Brunn-Minkowski type inequalities in~\cite{CE, CFM}. But it is mainly the more recent paper by Klartag~\cite{KlartagPTRF} in which he proves the variance conjecture in the case of unconditional distributions, that was the starting point of the present work. Before going on into this, let us discuss a bit the $L^2$ method itself.

Throughout the paper we shall work, unless otherwise stated, with a probability measure  $\mu$ on  $\R^n$ of the form 
\begin{equation}
\label{defmu}
d\mu(x) =e^{-\Phi(x)} \, dx, \qquad \textrm{with $\ \Phi:\R^n\to \R$ of class $C^2$}.
\end{equation}
%given by~\eqref{defmu} 
Introduce
the natural Laplace operator on $L^2(\mu)$ given by
\begin{equation}\label{def:L}
L u = \Delta u - \nabla \Phi \cdot \nabla u\, .
\end{equation}
This operator is well defined on $C^2$-smooth functions and can be extended into a unbounded closed self-adjoint operator on $L^2(\mu)$ with dense domain $\dom(L)$ corresponding to Neumann condition at infinity ensuring~\eqref{ipp1}; however, this domain is not important for our purposes  and we stress that it is enough to know that it contains the space of $C^2$-smooth functions that are compactly supported, which we denote by
$$\mathcal D := \{u:\R^n \to \R \; ;\ \textrm{$u$ of class $C^2$ and compactly supported}\}.$$

For $u\in \mathcal D$ and $f\in L^2(\mu)$ locally Lipschitz (this is just a sufficient requirement to perform the integration by parts, in virtue of Rademacher's differentiation theorem), we have
\begin{equation}\label{ipp1}
\int f Lu \, d\mu = - \int \nabla f \cdot \nabla u \, d\mu
\end{equation}
Since we work with  $\mu$ finite, the kernel of the  self-adjoint  operator  $L$ is given by the constant functions and its orthogonal space will be denoted by  
$$L_0^2(\mu):=\Big\{f \in L^2(\mu) \; ; \ \int f \, d\mu = 0 \Big\} 
%= \ker (L)^\perp
$$
The variance of $f$ is then the square of the $L^2$-norm of the projection of $f$ onto $L_0^2(\mu)$. Note that the Poincar\'e inequality~\eqref{defpoinc} for $\mu$ amounts to the spectral gap estimate $-L\ge \poinc{\mu}$ on $L_0^2(\mu)\cap \mathcal D$, say.
%In all the proofs, we will assume that for a smooth $f\in L^2(\mu)$    orthogonal to constant functions, i.e. $\int f \, \dmu =0$,  there exists a smooth $u\in L^2(\mu)$, unique up to an additive constant,  such that 
%$$f= Lu .$$ 
%This can ensured by assuming that $\Omega=\R^n$ and that $L$ has positive spectral gap. 
%This situation is much more general that the one we shall use in practice. 
%Actually, we could also work with the rather trivial observation (see~\cite[]) that the image by $L$ of smooth compactly supported function is $L^2(\mu)$-dense in the space of functions orthogonal to constant functions. 

The starting point of the argument is to dualize the Poincar\'e inequality using $L$. To this aim, many authors impose that one can solve $f=Lu$ for given $f\in L^2_0(\mu)$, which amounts to saying that $L$ has a closed range (oddly enough, this was rather the conclusion H\"ormander was aiming at). This has the disadvantage that one has to enter into tedious discussions and eventually impose further conditions on the measure $\mu$. 
%We are not willing to do so. 
Instead, we shall use a simple density argument. Let alone this point, the next Lemma is standard, except maybe for the fact that we have included a discussion on the invariances, for later use. So the classical and well-known statement corresponds to the case where $G=\{\id\}$, i.e. no invariance is imposed.

\begin{lemma}\label{lem:duality}
Let $\mu$ be a probability measure on $\R^n$ written as in~\eqref{defmu}, $L$ be the operator given by~\eqref{def:L} and  let $G$ a group of isometries leaving $\mu$ invariant. 
If there exists an application $A:x\to A_x$ from $\R^ n$ to the set of positive $n\times n$ matrices such that, for every $u\in \mathcal D$ that is  $G$-invariant we have 
 
 \begin{equation}
\label{dual:hyp}
\int \big(Lu \big)^2\, d\mu \ge \int A\nabla u\cdot \nabla u \, d\mu,
\end{equation}
 then for every $f\in L^2(\mu)$ locally Lipschitz that is $G$-invariant we have, 
 \begin{equation}
\label{dual:concl}
\var_\mu(f) \le \int A^{-1}\nabla f \cdot \nabla f\, d\mu.
\end{equation}
\end{lemma}

\begin{proof}
Let us fix $f\in L_0^2(\mu)$ that is locally Lipschitz and $G$-invariant,  and assume that we have the dual spectral inequality~\eqref{dual:hyp} for functions $u\in \mathcal D$ that are $G$-invariant.

In order to avoid discussion about solvability of $f=Lu$, we will use the following easy and classical fact, recalled in~\cite{CFM}: the space $L(\mathcal D)$ is dense in $L^2_0(\mu)$. We also need to check that one can use $G$-invariant functions $u$ to approach $f$. For this, note first that, for continuous functions,  invariance by $G$ or by the closure  of $G$ (in the usual topology of $\mathcal O_n$) is equivalent, so we can assume that $G$ is compact, and equipped with a bi-invariant Haar measure $\sigma$ normalized to be a probability. If $L u_k \to f$ in $L^2(\mu)$ for some sequence $u_k\in \mathcal D$, introduce $\tilde u_k := \int_G u_k\circ R \, d\sigma(R)$. Then $\tilde  u_k\in \mathcal D$ and by construction $\tilde  u_k$ is $G$-invariant. By convexity of the norm, using that $f$ and $\mu$ are $G$-invariant, we see that $L \tilde u_k \to f$ in $L^2(\mu)$, as wanted.

For an arbitrary function $u\in \mathcal D$ that is $G$ invariant, we have, using~\eqref{ipp1} and the assumption~\eqref{dual:hyp}:
\begin{eqnarray*}
\var_\mu(f) - \int (f-Lu)^2 \, d\mu & = & -2\int \nabla f \cdot \nabla u \, d\mu - \int (Lu)^2 \, d\mu  \\
 & \le & -2\int \nabla f \cdot \nabla u \, d\mu  - \int A\nabla u\cdot \nabla u \, d\mu \\
 & \le & \int A^{-1}\nabla f \cdot \nabla f\, d\mu.
\end{eqnarray*}
where we also used pointwise the inequality $2\,  v\cdot w \le A v \cdot v + A^{-1} w \cdot w$ for $v,w\in \R^n$. The conclusion then follows by the density argument recalled above.
\end{proof}

The power of the $L^2$ method relies on the fact that the dualization procedure of Lemma~\ref{lem:duality} allows for the use of the ``curvature" (or convexity) of the measure $\mu$ given by~\eqref{defmu}, which enters through the following classical integration by parts formula: for every $u\in \mathcal D$ we have
\begin{equation}\label{bochner}
\int (Lu)^2 \dmu = \int D^2\Phi(x) \nabla u(x) \cdot \nabla u(x) \dmu(x) + \int \left\| D^2 u (x) \right\|^2 \dmu(x)
\end{equation}
where $ \displaystyle \left\| D^2 u (x) \right\|^2:= \sum_{i,j\le n}(\partial^2_{ij} u(x))^2 $ is the square of the  Hilbert-Schmidt norm of the Hessian of $u$ at $x\in \R^n$.  In particular,  we see that $\int (Lu)^2 \dmu \ge \int D^2\Phi \nabla u \cdot \nabla  \dmu$, which translates,  when $\mu$ is strictly log-concave (meaning $D^2\Phi>0$), into  the celebrated Brascamp-Lieb inequality \cite{BL}: for every $f\in L^2(\mu)$ that is locally-Lipschitz 
$$ \var_\mu(f) \le \int \big(D^2\Phi\big)^{-1}\nabla f \cdot \nabla f \, d\mu.$$
In particular, if there exists $\rho \in (0,+\infty)$, such that $D^2 \Phi \ge \rho\,  \mathrm{Id}$ pointwise, then  $\poinc{\mu}\le \rho^{-1}$.

We can mention here  that when $\mu$ has some convex support $K$, i.e. $\mu$ has a density $e^{-\Phi}$ with  $\Phi$ of class $C^2$ on $K$, and equal to $+\infty$ outside $K$, then the integration by parts above incorporates a boundary term, which, by the convexity of $K$ is always nonnegative. Therefore, formula~\eqref{bochner} becomes an inequality $\ge$, which goes in the right direction for running all the arguments we use in the paper. This explains why our results can directly be extended to this more general class of measures. However, it seems like a challenging question to be able to use, in this case, the extra information coming from the boundary term.

Klartag used, among other things, a similar $L^2$ argument to provide an optimal bound for the variance of the function $f(x)=|x|^2$ when $\mu$ is a log-concave unconditional measure (actually, he worked with an unconditional convex body $K$).
% and thus with the measure $\mu_K$). 
Unconditionality means invariance with respect to the coordinate hyperplanes, i.e. under the reflections $S_{e_1^\perp}, \ldots, S_{e_n^\perp}$. Klartag manages to use~\eqref{bochner} in the form   $\int (Lu)^2\, d\mu\ge \int \| D^2 u \|^2\, d\mu$ when $\mu$ is log-concave. He then proves a ``$H^{-1}(\mu)$" estimate for the derivatives $\partial_i f$, which in turn is controlled using the Wasserstein (optimal transport) distance; a transportation argument closes the argument. 
%We shall extract from Klartag's argument its main new idea: a proper use of the invariance in the $L^2$ argument. 
%A slightly different proof will be given in the next section.

We want to emphasize the role of symmetry in the argument. The idea is that if $u$ has some invariance,
 then $\nabla u$ will have the ``anti-invariance" (for instance we pass from being even to being odd with respect to some direction, say). 
Let us see this principle in action in the unconditional case by giving a simplified version of  Klartag's proof. We shall work directly with measures instead of sets.

So let $\mu$ be be given by~\eqref{defmu} with $\Phi$ convex such that $\Phi(\pm x_1, \ldots, \pm x_n)= \Phi(x_1, \ldots, x_n)$. Let $u\in \mathcal D$ having the same invariances.
The argument relies on a lower bound for each $\int (\partial_{ii}^2 u)^2\dmu $. For notational simplicity, let us consider first the case $i=1$ and write $x=(x_1, y)$  for $x\in \R^n$ with $y=(x_2, \ldots ,x_n)$. For fixed $y \in \R^{n-1}$, the measure $e^{-\Phi(x_1, y)}\, dx_1$ is a (finite) log-concave measure on $\R$. Such measures are known to satisfy a Poincar\'e inequality; the bound~\eqref{KLS} for instance  yields that for  $v:\R\to \R$ smooth with $\int_\R v(x_1)\, e^{-\Phi(x_1,y)}\, dx_1 =0$, 
\begin{equation}\label{poinc1}
\int_\R v(x_1)^2 e^{-\Phi(x_1,y)} \, dx_1 \le 4 \frac{\int_\R x_1^2 e^{-\Phi(x_1,y)} \, dx_1}{\int_\R e^{-\Phi(x_1, y)} \, dx_1}\, \int_\R v'(x_1)^2 e^{-\Phi(x_1,y)} \, dx_1  
\end{equation}
%Let us remark that Klartag's proof assumed that 
%$\mu = \mu_K$, the Lebesgue measure restricted to an unconditional convex body. In this case, our previous discussion is essentially trivial since we are working with the %Lebesgue measure restricted to a symmetric segment  $[-b(y), b(y)]\subset \R$ for which spectral bounds are explicitly known and~\eqref{poinc1} is straightforward. 
%
But $u$ and $\Phi$  are unconditional, and so   the function $x_1 \to \Phi(x_1,y)$ is even and the function $x_1 \to \partial_i u(x_1,y)$ is odd, ensuring that $\int_\R  \partial_1 u(x_1 ,y)\, e^{-\Phi(x_1,y)} \, dx_1 = 0$. Then,~\eqref{poinc1} applies with $v(x_1)=\partial_1 u(x_1 ,y)$ and $v'(x_1)=\partial^2_{11} u(x_1,y)$.

To summarize, if  for  $i=1, \ldots , n$, we let $P_i$ be  the orthogonal projection onto the coordinate hyperplane $e_i^\perp$ and define
$$g_i(x) = g_i(P_i x):=4   \frac{\int_\R  x_i^2 \,  e^{-\Phi(P_i x + x_i e_i)} \, dx_i }{\int_\R e^{-\Phi(P_i x + x_i e_i)} \, dx_i} , $$
we have from the argument explained above for $i=1$ and from Fubini's theorem that
$$\int \frac1{g_i(P_i x)} \big( \partial_i u(x)\big)^2 \dmu(x) \le \int \big(\partial^2_{ii} u\big)^2 \dmu .$$
Using that $ \sum_{i=1}^n \big(\partial^2_{ii} u\big)^2 \le \|D^2 u \|^2$ in~\eqref{bochner}, we get a bound~\eqref{dual:hyp} which implies, by Lemma~\ref{lem:duality} the following estimate: for every function $f\in L^2(\mu)$ that is unconditional and locally Lipschitz, 
$$\var_\mu (f) \le \int  \sum_{i=1}^n g_i(P_i x)\,  \big(\partial_i f (x) \big)^2 \dmu(x).$$
When we apply this estimate to the particular case $f(x)=|x|^2$, the following happens. We can use 
 again Fubini's theorem and H\"older's inequality, to get that 
$$\var_\mu (|x|^2) \le 4 \sum_{i=1}^n  \int x_i^4 \dmu(x).$$
 It is a well known and useful consequence of the Pr\'ekopa-Leindler inequality due to Borell that $L_p$ norms of convex homogeneous functions with respect to a log-concave measure are equivalent \cite{Borell}.
 In particular, there exists a numerical constant $c>0$ such that for every $n\ge 1$, every log-concave probability  measure $\mu$  on $\R^n$ and every even semi-norm $H$ -- typically $H(x) = x\cdot \theta$ or $H(x)=|x|$ --  the following reverse H\"older inequality holds:
\begin{equation}\label{borell}
\int H(x)^4 \dmu(x) \le c \left(\int H(x)^2 \dmu(x)\right)^2.
\end{equation} 
We have therefore proved that for an unconditional  log-concave measure $\mu$ on $\R^n$ it holds that
$\var_\mu (|x|^2) \le \tilde C \sum_{i=1}^n  \big(\int x_i^2 \dmu(x))^2$. When $\mu$ is furthermore isotropic, i.e. with covariance matrix equal to the identity, the previous bound reads as
$$\var_\mu (|x|^2) \le  \tilde c \, n ,$$
which answers positively the variance conjecture in the case of unconditional distributions.

\section{Functions and measures with invariances}
\label{vargeneral}

As apparent from the treatment we gave of the unconditional case, we shall need to work with restrictions of measures onto subspaces. This idea has been used already been used successfully in statistical mechanics, and our goal is to bring the invariances into the game.

Let us start with some notation. 
%We equip $\R^n$ with its standard euclidean structure ; it induces a canonical euclidean structure on every affine subspace. 
For a subspace $F\subset \R^n$, we denote by $P_F$ the orthogonal projection onto $F$. Given a probability measure $\mu$ with density $e^{-\Phi}$ on $\R^n$, a subspace $E\subset \R^n$ and a point $x\in \R^n$, we denote by $\mu_{x,E}$ the  probability measure on $E$ obtained by conditioning $\mu$ to fixed $P_{E^\perp}x$, i.e. 
$$d\mu_{x, E} (y) := e^{-\Phi(y+P_{E^\perp}x)} \frac{dy}{\int_{E}e^{-\Phi(z+P_{E^\perp}x)} \, dz} , \qquad y\in E.$$
In other words, if  $X\sim \mu$, then
$$\mathbb E\big( X | \, P_{E^\perp} X = P_{E^\perp}x  \big)\;  \sim \mu_{x, E}.$$
The measure $\mu_{x,E}$ can be seen likewise as a measure on $E$ or on $x+E = P_{E^\perp}x+E$.
Note that $\mu_{x,E}$ depends only on $P_{E^\perp} x$: $\mu_{x,E} = \mu_{P_{E^\perp}x, E}$. For suitable $g:\R^n \to \R$, we shall extensively use Fubini's theorem 
in the form
\begin{equation*}
\int g\, d\mu = \int_{x \in E^\perp} \left(\int_{y\in E} g(x+y) \, e^{-\Phi(x+y)} \, dy \right) \, dx,
\end{equation*}
or in the form
\begin{equation}\label{fubini}
\int f \, d\mu= \int \left(\int f \, d\mu_{x,E}\right) \, d\mu(x).
\end{equation}

Let us recall a restriction argument put forward by Helffer~\cite{Helffer}. Here we work with a measure $\mu$ satisfying~\eqref{defmu} and the canonical basis $\{e_i\}$. Using  the Poincar\'e inequality for the measures $\nu=\mu_{x,\R e_i}$ in the form
 $\poinc{\nu} \int (L_\nu w)^2 d\nu \ge \int (w')^2 d\nu$ we have
 \begin{eqnarray*}
 \int (Lu)^2 d\mu &=& \int \big(\|D^2u\|^2+D^2\Phi \nabla u\cdot\nabla u\big)\, d\mu \\
 &\ge & \int \sum_{i\neq j} \partial^2_{i,j} \Phi \, \partial_i u \, \partial_j u \dmu+ \sum_i \int \int \Big((\partial^2_{i,i} u)^2+   \partial^2_{i,i} \Phi\, (\partial_i u )^2\Big)
d\mu_{x,E_i} \,d\mu(x) \\
&\ge & \int \sum_{i\neq j} \partial^2_{i,j} \Phi \, \partial_i u \, \partial_j u \dmu+ \sum_i \int \Big(\poinc{\mu_{x,\R e_i}}^{-1}\int (\partial_i u )^2
d\mu_{x,E_i} \Big)\,d\mu(x) \\
&=&  \int K\nabla u\cdot \nabla u \, d\mu.
\end{eqnarray*}
%\end{proof}
where the $n \times n$ matrix $K(x)$ is defined as follows: $K(x)_{i,i}= \poinc{\mu_{x,\R e_i}}^{-1}$ and for $i\neq j$, 
$K(x)_{i,j}=\partial^2_{i,j} \Phi (x)$. Therefore, by Lemma~\ref{lem:duality} (with no invariance yet, i.e. $G=\{\id\}$), we see that under the assumption that $K(x)>0$ for every $x\in \R^n$, we have the following inequality: for every $f\in L^2(\mu)$ that is locally Lipschitz, 
$$\var_\mu(f) \le \int K^{-1}\nabla f\cdot \nabla f \, d\mu.$$
Unlike in the Brascamp-Lieb inequality, log-concavity is not required here.
This is particularly effective for perturbations of product measures. 

We want to push forward this approach by working with functions sharing invariances with the underlying measure. We will need the following property of such functions,
which is obvious for a reflection (it then follows from the even/odd character of the functions).
\begin{fact}\label{fact:antisymmetry}
Let $\mu$ be a probability measure on $\R^n$, $R\in \O_n(\mu)$ and set $E=\fix(R)^\perp$. For every $R$-invariant function $g$ and for every $x \in \R^n$ we have,
$$\int P_{E}\nabla g(y) \, d\mu_{x,E}(y) = 0 .$$
In particular, the measure $\mu_{x,E}$ is centered.
\end{fact}
\begin{proof}
Denote by $\rho$ the density of $\mu$ and set $a:= \int_{E} P_E \nabla g(y) \, \rho(y) dy$. By definition $a\in E$. 
But since $E= \fix(R)^\perp$ and $R$ is normal,  we have $R E = E$ and $P_E R = RP_E$, and so
$$a = \int_{E} P_E \nabla g( Ry) \, \rho(Ry) dy = \int_{E} P_E R \nabla g(y) \, \rho(y) dy = Ra$$
were we used  that $\rho$ and $g$ are $R$-invariant.  This shows that $a \in \fix(R)= E^\perp$ and therefore $a=0$.
 \end{proof}

In the sequel,  
 we shall be interested in the case when we have fixed-point  subspaces $E_i=\fix(R_i)^ \perp$, $i=1,\ldots,m$,  that induce a decomposition of the identity of the form
\begin{equation}\label{decompid}
\sum_{i=1}^m c_i \, P_{E_i}  = \id
\end{equation}
where the $c_i$'s are positive reals. It will be used also in the form
$$\forall v \in \R^n, \qquad \sum_{i=1}^m c_i \, |P_{E_i} v |^2 \, = \, |v|^2 .$$
This situation naturally arises when we consider measures having enough invariances by reflections.
 The simplest example is when $m=n$, $E_i=\R e_i$ and $c_i=1$, where $\{e_i\}$ is the canonical basis of $\R^n$ (this corresponds to unconditional measures). 
 More examples appear in the appendix. Note that taking  traces yields $\sum c_i \text{dim}(E_i) = n$.

\begin{theorem}\label{theo:invariance1}
Let $\mu$ be a  probability measure on $\R^n$ given by~\eqref{defmu}. Assume there exists $R_1, \ldots, R_m \in \O_n(\mu)$ and $c_1, \ldots , c_m >0$ such that, setting  $E_i:=\fix(R_i)^\perp$, we have that $\{E_i, c_i\}$ decompose the identity in the sense of~\eqref{decompid}. 
Assume that for all $x\in \R^n$,
% in the sense of matrices 
$$ H(x):=D^2\Phi(x)+ \sum_{i=1}^m \frac{c_i}{ \poinc{\mu_{x,E_i}} }\, P_{E_i}>0.$$
Then for every locally Lipschitz and $\{R_i\}_{i\le m}$-invariant  function $f:\R^n\to \R$ we have,
$$ \var_\mu(f) \le \int H^{-1} \nabla f \cdot \nabla f \, d\mu.$$
In particular, if there exists $\rho\in\R$ such that $D^2\Phi\ge \rho \mathrm{Id}$  and for all $x,i$, $ \poinc{\mu_{x, E_i}}^{-1} + \rho\ge 0$, then  
 every for every locally Lipschitz and $\{R_i\}_{i\le m}$-invariant function $f:\R^n\to \R$,
\begin{eqnarray}
\label{eq:var2} 
\var_\mu(f) &\le & \int \left( \sum_{i=1}^m c_i \Big(\poinc{\mu_{x, E_i}}^{-1}+\rho\Big)^{-1} \,\big| P_{E_i} \nabla f (x)\big|^2 \right) \, d\mu(x)\\
&\le& \sup_{i,x} \Big(\poinc{\mu_{x, E_i}}^{-1}+\rho\Big)^{-1} \int |\nabla f |^2  d\mu. \nonumber
\end{eqnarray}
\end{theorem}
\begin{proof}
By Lemma~\ref{lem:duality}, it is sufficient to prove that for every $\{R_i\}$-invariant $u\in \mathcal D$, we have  $\int (Lu)^2 d\mu\ge \int H\nabla u\cdot\nabla u\, d\mu$, which rewrites, in view of~\eqref{bochner} as
\begin{equation}\label{eq:goal}
\int \|D^2u\|^2d\mu \ge \int \sum \frac{c_i}{\poinc{\mu_{x,E_i}}} |P_{E_i}\nabla u|^2 d\mu(x). 
\end{equation}
To establish the latter, first note that for every symmetric matrix $H$,
\begin{equation}\label{ineqHS}
\| H\|^2 =\mathrm{Tr}(H^2)= \sum_{i=1}^m c_i\, \mathrm{Tr}(P_{E_i} H^2 P_{E_i}) \ge \sum_{i=1}^m c_i \, \mathrm{Tr}((P_{E_i} H P_{E_i})^2)
 =\sum_{i=1}^m c_i \|P_{E_i} H P_{E_i}\|^2.
\end{equation}
It follows that $\int \|D^2u\|^2d\mu \ge \int \sum c_i  \|P_{E_i}D^2uP_{E_i}\|^2d\mu.$

Next, by Fact~\ref{fact:antisymmetry}, we know that for all $x,i$,  $\int P_{E_i}\nabla u \,d\mu_{x,E_i}=0$. Hence for all unit vector 
and $a\in E_i$, we deduce that the function
$h= \nabla u \cdot a$ verifies $\int h \, d\mu_{x,E_i}=0$ and  so $\int h^2 d\mu_{x,E}  \le \poinc{\mu_{x,E_i}}\int |P_{E_i} \nabla h|^2 \, d\mu_{x,E_i}$.
 By taking vectors $a$ forming an orthonormal basis of $E_i$ we deduce that
$$\int |P_{E_i} \nabla u |^2 \, d\mu_{x,E_i} \le \poinc{\mu_{x,E_i}} \int \| P_{E_i} D^2 u P_{E_i} \|^2 \, d\mu_{x,E_i} .$$
Using Fubini's theorem in the form of \eqref{fubini}, we get that for all $i=1,\ldots,m$,
$$\int  \| P_{E_i} D^2 u P_{E_i} \|^2  \, d\mu\ge \int \frac1{\poinc{\mu_{x,E_i}}}  |P_{E_i} \nabla u (x)|^2 \, d\mu(x).$$
Summing upon the index $i$ gives \eqref{eq:goal}, which concludes the proof of the general case.

\medskip
For the special case when $D^2\Phi\ge \rho \mathrm{Id}$, we  have  $H(x) \ge \sum c_i (\rho + \poinc{\mu_{x,E_i}}) P_{E_i}$ and the result follows by bounding $H^{-1}$ from above. Indeed, if a positive matrix verifies $H\ge \sum c_i \alpha_i P_{E_i}$ with the $\alpha_i>0$, then,
for every vector $v\in \R^n$, setting $w= H^{-1} v$, we have
$$
H^{-1} v \cdot v  =  2 v\cdot w - H w \cdot w 
\le \sum_{i=1}^m c_i \Big( 2 P_{E_i} v \cdot P_{E_i} w - \alpha_i |P_{E_i} w|^2 \Big)  
  \le   \sum_{i=1}^m \frac{c_i}{\alpha_i} |P_{E_i} v|^2.
$$
%
%could estimate $H^{-1}$ from above.
%Alternatively, note that the above argument gives 
%\begin{eqnarray*}
% \var_\mu(f) &=& \int (Lu)^2\, d\mu \ge \int \left[\rho |\nabla u|^2 +  \sum \frac{c_i}{\poinc{\mu_{x,E_i}}} |P_{E_i}\nabla u|^2\right] d\mu(x) \\
%  &=& \int \left[ \sum c_i \Big(\poinc{\mu_{x,E_i}}^{-1} +\rho\Big)\,  |P_{E_i}\nabla u|^2\right] d\mu(x),
%\end{eqnarray*}
%where the last equality comes from the decomposition of the identity. Next we dualise this estimate
%\begin{eqnarray*}
%\var_\mu(f) &=& \int f\, Lu \, d\mu= -\int \nabla f \cdot \nabla u= -\int \sum_{i=1}^m c_i  P_{E_i} \nabla f \cdot P_{E_i} \nabla u \, d\mu  \\
%  &\le & \left( \int \sum_{i=1}^m c_i \Big(\poinc{\mu_{x, E_i}}^{-1}+\rho\Big)^{-1} \,\big| P_{E_i} \nabla f (x)\big|^2  \,d\mu(x)\right)^{\frac{1}{2}}\left(\var_\mu(f)  \right)^{\frac{1}{2}}.
% \end{eqnarray*}
\end{proof}

Next, we present a variant of the previous result, which does not require invariances.% (as indicated by Fact~\ref{fact:antisymmetry}). 

\begin{theorem}\label{theo:invariance2}
Let $\mu$ be a  probability measure on $\R^n$ given by~\eqref{defmu}
 such that $D^2\Phi\ge \rho \mathrm{Id}$ on $\R^n$ for some  $\rho\ge 0$.
Let $E_1, \ldots, E_m$ be subspaces of $\R^n$ and  $c_1, \ldots , c_m >0$ such that $\{E_i, c_i\}$ decompose the identity in the sense of~\eqref{decompid}. Then, for every locally Lipschitz function $f:\R^n\to \R$ such that for every $x\in \R^n$ and $i\le m$,
\begin{equation}\label{zeromean}
\int P_{E_i} \nabla f  \, d\mu_{x,E_i} = 0,
\end{equation}
we have
$$\var_\mu(f) \le \int \left( \sum_{i=1}^m c_i \Big( \poinc{\mu_{x, E_i}}^{-1}+\rho\Big)^{-1}  | P_{E_i} \nabla f (x)|^2 \right) \, d\mu(x).$$
\end{theorem}

\begin{proof}
Here, we will not use Lemma~\ref{lem:duality} directly since our hypothesis does not translate into a property of $u$. So we will run again the $L^2$ duality argument and, unlike previously, use an exact solution to $f=Lu$.
This causes no problem, as the measure $\mu$ is now log-concave and \eqref{KLS} ensures that $L$ has
 a spectral gap. Hence,  for $f\in L^2_0(\mu)$ we can find a $u$ (in the domain of $L$) such that $f=Lu$.
Then, on one hand
\begin{eqnarray*}
 \var_\mu(f) &=& \int (Lu)^ 2 d\mu  \ge  \int \Big(\|D^2 u\|^2+ \rho |\nabla u|^2\Big) \, d\mu  \\
  &\ge & \sum_i c_i \int \Big(\|P_{E_i}D^2u P_{E_i}\|^2 +\rho |P_{E_i} \nabla u |^2 \Big) d\mu  \\
  &\ge&  \sum_i c_i \left[ \int  \poinc{\mu_{x,E_i}}^{-1} \left(\int \Big|P_{E_i}\nabla u -\int P_{E_i}\nabla u \, d\mu_{x,E_i}\Big|^2 d\mu_{x,E_i}\right)\,  d\mu(x)\right. \\
 && \left. +\rho \int \left( \int |P_{E_i}\nabla u -\int P_{E_i}\nabla u \,d\mu_{x,E_i}|^2 \,d\mu_{x,E_i} \right) d\mu(x)  \right]\\
  &=& \int \sum_i c_i \Big( \poinc{\mu_{x,E_i}}^{-1} +\rho \Big) \, \Big|P_{E_i}\nabla u -\int P_{E_i}\nabla u \, d\mu_{x,E_i}\Big|^2 d\mu(x),
\end{eqnarray*}
where we have used the Poincar\'e inequality for $\mu_{x,E_i}$ and the inequality $\int g^2 d\mu_{x,E_i} \ge \var_{\mu_{x,E_i}}(g)$.
On the other hand, using now the hypothesis $\int P_{E_i} \nabla f  \, d\mu_{x,E_i} = 0$ gives 
\begin{eqnarray*}
\var_\mu(f) &=&  \int f Lu \dmu = - \int \nabla f \cdot \nabla u \, d\mu = 
-\int  \sum_{i=1}^m c_i P_{E_i} \nabla f \cdot P_{E_i} \nabla u   \, d\mu \\
&=& -\int  \sum_{i=1}^m c_i P_{E_i} \nabla f \cdot \Big( P_{E_i} \nabla u-\int P_{E_i} \nabla u\, d\mu_{x,E_i} \Big)   \, d\mu(x).
\end{eqnarray*}
The claim then follows from the Cauchy-Schwarz inequality.
\end{proof}

\begin{remark}
The following fact was crucial in the proof of the $B$-conjecture~\cite{CFM}: if $D^2\Phi\ge \mathrm{Id}$,
 then any function with $\int \nabla f \, d\mu=0$ verifies $\var_\mu(f)\le \frac12 \int |\nabla f|^2 d\mu$. The novelty there was of course the improved $1/2$ factor, which can be seen as a ``second eigenvalue" estimate. 
 Applying the previous theorem with $m=1$, $E=\R^n$, $\rho=1$ recovers it.
  Indeed the  hypothesis on the Hessian  of $\Phi$ (which implies
 $\poinc{\mu}\le 1$) gives by restriction a similar inequality of the Hessian of the potential of $\mu_{x,E_i}$, which guarantees $\poinc{\mu_{x,E_i}}\le 1$. 
\end{remark}

%%%%%%%%%%%%%%%%%%%%%%%%%%%%%%%%%%%%%%%%%%%%%%%%%%

\section{Variance of the norm for log-concave measures}
\label{varnorm}

We present here a first application of the previous result to the study of the variance conjecture~\eqref{varconj} for measure with invariances. Indeed,
the function $f(x)=|x|^2$ is invariant by all isometries, so we can apply the bound of the previous section when $\mu$ (resp. $X\sim \mu$) is a log-concave measure (resp. a log-concave random vector)  of $\R^n$ have well distributed invariances. 
Introduce the quantity
$$ v(n):=\sup\big\{ \mathrm{Var}(|X|^2);\; X \mbox{ isotropic log-concave random vector in }\R^n\big\}.$$
Recall that the variance conjecture predicts that $v(n)\le c\, n$ for some universal constant $c$. The Cauchy-Schwarz inequality
immediately yields $v(n)\le  n^2$. 
Improving this trivial estimate is a difficult problem, recently solved 
by several authors, which led to the solution of the so-called central limit theorem for convex bodies. 
Up to date, the best known estimate is due to Gu\'edon and E. Milman~\cite{Guedon:2011wl}: $v(n)\le C \, n^{5/3}$.
% GM bound : $n^{5/3}= n \times (n^{1/3})^2

%Note that, by theCauchy-Schwarz inequality, we have, for every isotropic log-concave vector $X$ in an Euclidean space $(E,|\cdot|)$, the estimate
%\begin{equation}
%\label{trivialbound}
%\var(|X|^2) \le \textrm{dim}(E)^2
%\end{equation}
%Improving this trivial estimate is a difficult problem, recently solved 
%by several authors, which led to the solution of the so-called central limit theorem for convex bodies. Up to date, the best known estimate is  $\var(|X|^2) \le \textrm{dim}(E)^2 C n^{2-1/8}$. The variance conjecture remains open, expect for the case of unconditional bodies. 

\begin{theorem}\label{th:varnorm}
Let $X$ be a log-concave random vector in $\R^n$. Assume that there exist isometries $R_1,\ldots, R_m \in \O_n(X)$ and numbers $c_1,\ldots,c_m\ge 0$ such that, setting  $E_i=\fix(R_i)^\perp$, we have that $\{E_i, c_i\}$ decompose the identity in the sense of~\eqref{decompid}.   Then,
$$
\var(|X|^2)  \le  16\sum_{I=1}^m  c_i \, \dE |P_{E_i} X|^4, 
$$
and when $X$ is also isotropic,
$$ \var(|X|^2) \le C\, \sum_{i=1}^m c_i \, d_i^2 \,  \le C \, n \max_{i\le m} d_i$$
where $d_i:= \text{dim}(E_i)$. Here  $C>0$ is some numerical constant.
\end{theorem}

\begin{proof}
Let us denote $F_i=E_i^\perp$.  Theorem~\ref{theo:invariance1} applied to $f(x) = |x|^2$ and $X\sim \mu$ gives
$$ \var_\mu(f)\le 4\int \sum_i c_i \poinc{\mu_{x,E_i}} \, |P_{E_i}x|^2 d\mu(x).$$
Next, by the bound~\eqref{KLS} applied to the log-concave measure $\mu_{x,E_i}$ (for fixed $i,x$), we have
%VALEUR DE $\kappa$?
%Switching to probabilistic notation, we obtain that
 $$\poinc{\mu_{x,E_i}} \le 4 \int_{E_i} |y|^2d\mu_{x,E_i}(y)= 4\dE\left[|P_{E_i} X| \, \Big| \, P_{F_i}X=  P_{F_i}x \right].$$
Conditioning on $P_{E_i} X$ and applying Cauchy-Schwarz inequality, yields
\begin{eqnarray*}
\var(|X|^2) &\le& 16\sum_{i=1}^m c_i \dE\left[ \dE\Big[ | P_{E_i} X|^2\,  \Big| P_{F_i} X \Big]\,  \big|P_{E_i} X \big|^2 \right] \\
&\le & 16\sum_{i=1}^m c_i \dE\left[ \dE\Big[ | P_{E_i}X|^4\,  \Big| P_{F_i} X\Big]\right] 
=c\sum_{i=1}^m c_i \dE\left[ | P_{E_i}X|^4 \right] .
\end{eqnarray*}
The second inequality follows from Borell's lemma~\eqref{borell} and $\dE |P_EX|^2 = \text{dim}(E)$ when $X$ is isotropic.
The last one is derived from the relation $n=\sum_i c_i d_i$.
\end{proof}

Actually, one can relax the hypothesis on the structure of the isometries, by allowing invariant directions.

\begin{theorem}
\label{theo:var}
Let $X$ be an isotropic log-concave random vector in $\R^n$. Assume that there exist isometries $U_1,\ldots, U_m\in \O_n(X)$ and numbers $c_1,\ldots,c_m\ge 0$ such that
$$\sum_{i=1}^m c_i \, P_{\mathrm{Fix}(U_i)^\bot} =P_E,$$
where $E=\big( \cap_i \mathrm{Fix}(U_i)\big)^\bot$. Set $d= \mathrm{dim} \big(\cap_i \mathrm{Fix}(U_i)\big)$ and 
$d_i=\mathrm{codim}\big(\mathrm{Fix}(U_i)\big)$.
Then $$\mathrm{Var}(|X|^2)\le 2 v(d)+ c n \max_i d_i.$$
In particular, if $d\le \alpha \sqrt{n}$ and for all $i$, $d_i\le \alpha$, then $\mathrm{Var}(|X|^2)\le C(\alpha)n$.
\end{theorem}
\begin{proof}
Write $X=(Y,Z)\in E\times E^\bot$. Then $Y$ and $Z$ are log-concave isotropic random vectors in $E$ and $E^\bot$.
Since $E^\bot=\cap_i \mathrm{Fix}(U_i)$, $X\sim U_iX= (U_iY,Z)$ so $U_iY\sim Y$. Since 
$E\cap   \cap_i \mathrm{Fix}(U_i)=\{0\}$, we may apply the previous statement to $Y\in E$ and get 
$\mathrm{Var}(|Y|^2)\le  c (n-d) \max_i d_i.$ For $Z\in E^\bot$ we apply the trivial estimate 
$\mathrm{Var}(|Z|^2)\le v(d) \le d^2$. Eventually $\mathrm{Var}(|X|^2)= \mathrm{Var}(|Y|^2+|Z|^2)\le 
2 \mathrm{Var}(|Y|^2)+ 2 \mathrm{Var}(|Z|^2).$
\end{proof}

The preceding results are particularly useful in the case we have nice invariances by reflections, since $d_i=1$ then. Indeed, if $\fix(\mathcal R(X))=\{0\}$, then we can find a decomposition of the identity by directions orthogonal to the hyperplane symmetries, and therefore we get the desired bound $\var(|X|^2)\le c \, n$. Actually, a little more is true, in the spirit of the previous result.
%Using basic fact about reflection groups gives the following statement. 
%If $H$ is a linear hyperplane, we denote by $S_H$ the corresponding reflection.
\begin{theorem}
Let $X$ be an isotropic log-concave random vector in $\R^n$. Assume that there exist reflections
$S_{H_1},\ldots, S_{H_m} \in \O_n(X)$ with $\mathrm{dim} \big( \bigcap_i H_i)\le \alpha \sqrt n$.
Then  $\mathrm{Var}(|X|^2)\le C(\alpha)n$.

In particular, if $\bigcap_i H_i= \{0\}$, the vector $X$ verifies the variance conjecture.

\end{theorem}

\begin{proof}
This relies on basic fact about reflection groups recalled in the appendix. More precisely, the result follows from Lemma~\ref{lem:dec-reflex} and from the previous Theorem.
\end{proof}

The previous statement gives that every isotropic log-concave distribution which has the invariance of the simplex satisfies the variance conjecture.

Let us present another application, which does not involve  reflections. 
Consider $S_{p}^d$ the Schatten class, i.e. the space of $d\times d$ real matrices equipped with the norm $\|A\|_p^p  = \sum_{i=1}^d \lambda_i(A)^p$ for
 $p\in [1,+\infty)$, where the $\lambda_i(A)$ denote the singular values of $A$, i.e. the eingenvalues of $A^\ast A$.
Consider the linear applications $R_{i}$ which flip the signs of all the entries in the $i$-th row of a matrix.
 Clearly $\|R_i A \|_p=\| A\|_p$. Moreover $\mathrm{Fix}(R_i)$ is of dimension $d$ and $\sum_{i=1}^d P_{\mathrm{Fix}(R_i)^\bot}=\mathrm {Id}$ (see the Appendix
 for more details).  Let $n=d^2$ be the dimension and let  $B_{p}^d\subset \R^n$ be the unit ball for the  Schatten norm $\|\cdot\|_p$. Consider a suitable dilation $\lambda B_p^d$ ensuring that the random vector $X_{p,n}\sim \mu_{\lambda B_p^d}$ uniformly distributed on $\lambda B_p^d$ is isotropic. Theorem~\ref{th:varnorm} then gives that
$$\var(|X_{n,p}|^2) \le c \, n  d= c \, n^{3/2}$$
which is slightly better than the general bound of Gu\'edon and E. Milman.
%, but does not answer the variance conjecture for the Schatten classes.

%%%%%%%%%%%%%%%%%%%%%%%%%%%%%%%%%%%%%%%%%%%%%%%%%%%%%%%%%%%
%%%%%%%%%%%%%%%%%%%%%%%%%%%%%%%%%%%%%%%%%%%%%%%%%%%%%%%%%%%

\section{Invariant measures and general functions}
\label{sec:symmetrisation}
The goal of this section is to get rid of the invariance hypotheses  for the functions in Theorem~\ref{theo:invariance1}. 
We will prove that when some  group of isometries $G$ leaving the measure invariant  has nice spectral properties, 
%(as it is in the example we have in mind), 
then indeed  Theorem~\ref{theo:invariance1} is valid for all functions. For this, we will average (or symmetrize) through the group $G$. This method is inspired in part by an argument of B.~Fleury \cite{Fleury} who treated the case of unconditional measures.
Let us fix the setting.

\begin{setting} \label{setting}
Let $\mathcal G = \{R_1, \ldots , R_m\}$ be a set of $m$ isometries and let $G$ be the (compact) group they generate, equipped with its (normalized) Haar measure $\gamma$. We assume that 
\begin{itemize}
\item The set of generators $\mathcal G$ is stable under conjugacy in $G$ (i.e. $g\mathcal G g^{-1} = \mathcal G,\ \forall g\in G$).

\item We consider the  Cayley graph associated to these generators and we suppose that some Poincar\'e inequality holds on $G$ for some (generalized) discrete gradient. More precisely, we assume that there exists positive numbers $d_i$ such that: 
%for every $f:G\to \R$,
\begin{equation}
\label{poincG}
\forall f:G\to \R, \quad \var_\gamma(f) \le \mathcal  E(f,f):= \int \sum_{i=1}^m d_i\, \big[ f(gR_i) - f(g)\big]^2 \; d\gamma(g).
\end{equation}
We impose the following structural condition on the weights $d_i$:
\begin{equation}
\label{structcond}
\forall i,j\le m, \, \forall g\in G: \quad \ R_i=g\, R_j g^{-1} \Longrightarrow d_i=d_j 
\end{equation}
This condition is obviously satisfied if we take all the $d_i$ to be equal to the same constant $c>0$, and then the best constant $c$ for which~\eqref{poincG} holds is known to be the Poincar\'e constant associated to $(\mathcal G, G)$ which we denote by $\poinc{\mathcal G}$.
\end{itemize}
\end{setting}

%The reason we consider the possibility of different weights $d_i$ satisfying~\eqref{structcond} instead of just $\poinc{\mathcal G}$ is that we want to take advantage 
% of a possible product structure of the group $G$, each block coming with its own, possibly distinct, Poincar\'e constant (see below).

A crucial observation that follows  from the fact that $\mathcal G$ is stable under conjugacy, is that $G$ acts as a permutation not only on $\mathcal G$, but also on the set of fixed subspaces $E_i:= \fix(R_i)^\perp$, in the sense that for every $g\in G$ there exists a permutation $\tau$ of $\{1, \ldots , m\}$ such that
$$\big(g E_1 , gE_2, \ldots , g E_m\big)=\big(E_{\tau(1)}, E_{\tau(2)}, \ldots, E_{\tau(n)}\big).$$
To check this, just note that $\mathrm{Fix}(gR_ig^{-1})^ \bot= g (\mathrm{Fix}(R_i)^\bot)=gE_i$. 
%By the way, note that, in view of~\eqref{structcond}, we also have that  $g E_i = E_j \Rightarrow d_i=d_j$.

Let us first state a very simple but crucial observation:
\begin{fact}\label{meanzero}
Let  $u:\R^n\to \R$ be an arbitrary function and $R$ be an isometry leaving a measure $\mu$ invariant. Then, setting  $E= \fix(R)^\perp$, we have that for  every  $x\in \R^n$,
$$\int \big[u\circ R - u\big]\, d\mu_{x,E}  \; = \; 0.$$
\end{fact}
\begin{proof}
Set $F=\fix(R)$ and let $\rho$ be the density of $\mu$. As for Fact~\ref{fact:antisymmetry}, we use that $RE = E$. For every fixed $x\in \R^n$ we have, since $\rho\circ R= \rho$,
$$\int_{E} [(u\circ R)\cdot \rho](P_Fx+y) \, dy = \int_E   [(u\circ R)\cdot (\rho\circ R)](P_{F} x + y) \, dy = \int_{E} [u\cdot \rho] (RP_{F} x + z) \, dz $$
where we used that $R$ is an isometry of $E$ and performed the change of variable $z=Ry$.
To conclude, use that $R P_{F} = P_{F}$.
% we conclude that the latter integral is equal to $\int_{E} u \, \rho(P_{F} x + y) \, dy$, as stated.
\end{proof}

We are ready to start the extension of variance estimates from invariant to non-invariant functions.
The first step is to estimate from above the variance of a general function by the variance of the $G$-invariant function 
obtained by averaging.
\begin{proposition}\label{prop:var}
Under the Setting~\ref{setting} with $E_i:=\fix(R_i)^\perp$ we have, for every locally Lipschitz function $f$ on $\R^n$ 
 and $F:= \int f\circ g \, d\gamma(g)$,
%Then
$$\var_\mu(f) \le \var_\mu(F)+4 \int \sum_i d_i\, \poinc{\mu_{x,E_i}} \, |P_{E_i}\nabla f(x)|^2 d\mu(x).$$
\end{proposition}

\begin{proof}
Let $X$ be a random vector with distribution $\mu$ and $\Gamma$ a random isometry of distribution $\gamma$, chosen independently of $X$.
Then by the invariance properties of $\mu$, the random vector $\Gamma X$ has distribution $\mu$.
Hence, using the classical decomposition of the variance 
\begin{eqnarray*}
 \var_\mu(f)&=& \var f(\Gamma X) =\var_X\left( \mathbb E_\Gamma f(\Gamma X)\right) + \mathbb E_X \left(\var_\Gamma f(\Gamma X)\right)\\
 &=& \var_\mu(F) + \int \var_\gamma(g\mapsto f(gx)) \, d\mu(x).
\end{eqnarray*}
We estimate the second term by the discrete Poincar\'e inequality~\eqref{poincG} on $G$:
\begin{eqnarray*}
 \int \var_\gamma(g\mapsto f(gx)) \, d\mu(x)&\le & \int \int_G \sum_i d_i\, \big(f(gR_ix)-f(gx) \big)^ 2 d\gamma(g) \, d\mu(x)\\
 &=& \int_G \sum_i d_i\, \int_{\R^ n} \big(f\circ g(R_ix)-f\circ g(x) \big)^ 2 d\mu(x) \, d\gamma(g)
\end{eqnarray*}
The crucial point is that,  setting $h=f\circ g$, we have in view of Fact~\ref{meanzero}
that  for all $x$,
$ \int_{E_i} (h\circ R_i -h) \, d\mu_{x,E_i}=0$,
which allows to apply the Poincar\'e inequality for the restrictions of $\mu$ to subspaces parallel to $E_i$, without
a remainder term, i.e. in the form $$\int_{E_i} (h\circ R_i -h)^ 2 \, d\mu_{x,E_i} \le \poinc{\mu_{x,E_i}} \int_{E_i} |P_{E_i} \nabla (h\circ R_i -h)|^2 d\mu_{x,E_i}.$$
Consequently, using Fubini's theorem,
\begin{gather*}
\int_{} \big(f\circ g(R_ix)-f\circ g(x) \big)^ 2 d\mu(x) \le  
\int \poinc{\mu_{x,E_i}} \int_{E_i} |P_{E_i}\nabla (h\circ R_i)-P_{E_i}\nabla h |^ 2 d\mu_{x,E_i} \, d\mu(x)\\
\le 2\int \poinc{\mu_{x,E_i}} \left(|P_{E_i}\nabla (f\circ gR_i)(x)|^2+|P_{E_i}\nabla  (f\circ g)(x) |^ 2 \right)\, d\mu(x)
\end{gather*}
Next, observe that $|P_{E_i}\nabla  (f\circ g) |=|P_{E_i}g^ {-1}(\nabla  f)\circ g) |=|gP_{E_i}g^ {-1}(\nabla  f)\circ g) |=
|P_{gE_i}(\nabla  f)\circ g) |$.
Hence, using the invariance of $\mu$ by $G$ twice (first in the form $\poinc{\mu_{x,E_i}}=c_P(\mu_{gx,gE_i})$) and using the fact that
$G$ acts as a permutation on $(E_1,\ldots,E_m)$
\begin{eqnarray*}
\lefteqn{\sum_i\int_G\int \poinc{\mu_{x,E_i}} |P_{E_i}\nabla  (f\circ g)(x) |^ 2 d\mu(x) \, d\gamma(g)}\\
&=& \sum_i\int_G\int \poinc{\mu_{gx,gE_i}} |P_{gE_i}\nabla  f (gx) |^ 2 d\mu(x) \, d\gamma(g)\\
&=& \int_G\sum_i\int \poinc{\mu_{y,gE_i}} |P_{gE_i}\nabla  f (y) |^ 2 d\mu(y) \, d\gamma(g)\\
&=& \int_G\sum_j\int \poinc{\mu_{y,E_j}} |P_{E_j}\nabla  f (y) |^ 2 d\mu(y) \, d\gamma(g)=\sum_i\int \poinc{\mu_{y,E_i}} |P_{E_i}\nabla  f (y) |^ 2 d\mu(y) 
\end{eqnarray*}
By a similar calculation, using that that $R_iE_i=E_i$ 
%(since $E_i=\mathrm{Fix}(R_i)^\bot$ and $R_i$ is an isometry).
\begin{eqnarray*}
\lefteqn{\sum_i\int_G\int \poinc{\mu_{x,E_i}} |P_{E_i}\nabla  (f\circ gR_i)(x) |^ 2 d\mu(x) \, d\gamma(g)}\\
%&=& \sum_i\int_G\int c_P(\mu_{gR_ix,gR_iE_i}) |P_{gR_iE_i}\nabla  f (gR_ix) |^ 2 d\mu(x) \, d\gamma(g)\\
&=& \int_G\sum_i\int \poinc{\mu_{y,gR_iE_i}} |P_{gR_iE_i}\nabla  f (y) |^ 2 d\mu(y) \, d\gamma(g)\\
&=& \int_G\sum_j\int \poinc{\mu_{y,E_j}} |P_{E_j}\nabla  f (y) |^ 2 d\mu(y) \, d\gamma(g)=\sum_i\int \poinc{\mu_{y,E_i}} |P_{E_i}\nabla  f (y) |^ 2 d\mu(y) 
\end{eqnarray*}
Combining the above inequality gives the claim.
\end{proof}

If a variance estimate of Sobolev type 
 is available for $G$-invariant functions, it may be applied to the  above function $F$. The terms involving the gradient of 
$F$ have to be estimated in terms of the gradient of the initial function $f$. To do this, we first introduce a natural definition.
We shall say that a map from $\R^n$ to the set of quadratic forms on  $\R^n$, $Q: x\mapsto Q_x$ is invariant by a group $G$ of isometries of 
$\R^n$ if for all $x,u\in\R^n$ and all $g\in G$,
$$ Q_{gx}(gu)=Q_x(u).$$

\begin{lemma}\label{lem:grad}
Let $G$ be a group of isometries equipped with its Haar probability measure $\gamma$.
Let $\mu$ be a $G$-invariant probability measure on $\R^n$, and $Q$ be a $G$-invariant function
on $\R^n$ with values in positive quadratic forms.
Then for all  $f:\R^n \to \R^n$ with gradient in $L^ 2(\mu)$, setting $F= \int_G f\circ g \, d\gamma(g)$, it holds
$$ \int Q_x(\nabla F(x)) \, d\mu(x) \le \int Q_x(\nabla f(x)) \, d\mu(x) .$$
\end{lemma}
\begin{proof}
Since for all $x$, $u\mapsto Q_x(u)$ is convex, 
\begin{eqnarray*}
\int Q(\nabla F) \, d\mu &= & \int Q_x\left(\int g^{-1} (\nabla f) (gx) \, d\gamma(g) \right)\, d\mu(x)\\
&\le & \int\int  Q_x\left( g^{-1} (\nabla f) (gx)  \right)\, d\mu(x)\, d\gamma(g)\\
&=& \int\int  Q_{g^{-1}y}\left( g^{-1} (\nabla f) (y)  \right)\, d\mu(y)\, d\gamma(g)\\
&=& \int\int  Q_{y}\left(  \nabla f (y)  \right)\, d\mu(y)\, d\gamma(g)=\int Q(\nabla f)\, d\mu,
\end{eqnarray*}
where we have use the invariance of $\mu$ and then the one of $Q$.
\end{proof}

\begin{lemma}\label{lem:Qinv}
Let $G$ be a group of isometries of $\R^n$.
\begin{enumerate}
\item If $\Phi:\R^ n \to \R^ n$ is twice differentiable and $G$-invariant then $D^2\Phi$ is also $G$-invariant (in the sense of 
  quadratic forms valued functions).
\item Let $(E_1,\ldots,E_m )$ be a $m$-tuple of subspaces of $\R^n$, onto which $G$ acts by permutation>
%(ie for all $g,i$ there exists $j$ such that $gE_i=E_j$). 
Let $c(x,E_i)$ be coefficients such that for all $x,g,i$, $c(gx,gE_i)=c(x,E_i)$. Then (identifying quadratic forms and symmetric linear maps),
%via the canonical Euclidean structure), 
$x\mapsto \sum_i c(x,E_i) P_{E_i}$
is also $G$-invariant.
\item If $x\mapsto Q_x$ is $G$-invariant with values in definite positive quadratic forms, then so is the dual form map $x\mapsto Q_x^*$.
\end{enumerate}
\end{lemma}
\begin{proof}
The first item is obvious by differentiating $\Phi(gx)=\Phi(x)$.
For the second item, denoting the quadratic form by $Q_x$, and using the properties of $c$ gives
\begin{eqnarray*}
Q_{gx}(gu)&=& \sum_i c(gx,E_i) |P_{E_i}gu|^2 = \sum_i c(gx,E_i) |g^{-1}P_{E_i}gu|^2\\
  &=& \sum_i c(x,g^{-1}E_i) |P_{g^{-1}E_i}u|^2= \sum_j c(x,E_j) |P_{E_j}u|^2=Q_x(u), 
\end{eqnarray*}
where in the last equality, we have used that $g^{-1}$ induces a permutation of $(E_1,\ldots,E_m)$.
The last item is straightforward from the definition of the dual of a definite positive quadratic form
$Q^*(x)=\sup\{ (x\cdot y)^2/Q(y);\; y\neq 0\}$.
Also, recall that if $Q(x)=Ax\cdot x$ for all $x$, then $Q^*(x)=A^{-1}x\cdot x$.
\end{proof}

We can now state the extensions of  the estimates of Section~\ref{vargeneral} to general functions.

\begin{theorem}[Extension of Theorem~\ref{theo:invariance1} to general functions]\label{th:vargeneral}
Let $d\mu(x)=e^{-\Phi(x)}dx$ be a probability measure which is invariant by isometries $R_1,\ldots,R_m$ and set  $E_i=\mathrm{Fix}(R_i)^\bot$. Assume the conditions of Setting~\ref{setting} are satisfied.
Let $c_1,\ldots,c_m>0$ be coefficients also satisfying the condition~\eqref{structcond}  and such that $\{E_i,c_i\}$ decompose the identity in the sense of~\eqref{decompid}. 
%Assume Conditions~\eqref{condaction} are verified. 
%\medskip

If pointwise, $H_x:=D^2\Phi(x)+\sum_i \frac{c_i}{\poinc{\mu_{x,E_i}}} P_{E_i}>0$,
then for all $f:\R^n\to \R$ locally Lipschitz we have
 $$ \var_\mu(f)\le \int \Big(H_x^{-1}+4\sum d_i \poinc{\mu_{x,E_i}}P_{E_i}\Big)\nabla f(x) \cdot \nabla f(x) \, d\mu(x).$$

In particular, if there exists $\alpha\ge 0$ such that for all $x$, $D^2\Phi(x)\ge -\alpha \mathrm{Id}$ and for all $i$, $\poinc{\mu_{x,E_i}}<1/\alpha$,
then for all $f:\R^n\to \R$ locally Lipschitz,
 $$ \var_\mu(f)\le \int \sum_i (c_i+ 4d_i) \big(\poinc{\mu_{x,E_i}}^{-1}-\alpha \big)^{-1} |P_{E_i} \nabla f(x)|^2 \, d\mu(x).$$
 
 These two inequalities hold in particular if we take $d_i\equiv \poinc{\mathcal G}$, the Poincar\'e constant of the group.
\end{theorem}

\begin{proof}
Given $f$, the $G$-invariant function $F=\int f\circ g \, d\gamma$ satisfies $\var_\mu(F)\le \int H^{-1}\nabla F\cdot \nabla F \, d\mu$ thanks to 
Theorem~\ref{theo:invariance1}. Next, Lemma~\ref{lem:Qinv} ensures that $x\to H_x^{-1}$ is $G$-invariant. Thus, by Lemma~\ref{lem:grad},
$\int H^{-1}\nabla F\cdot \nabla F \, d\mu$ is at most $\int H^{-1}\nabla f\cdot \nabla f \, d\mu$. The first claim thus follows from
Proposition~\ref{prop:var}.
  
  For the second inequality, we apply~\eqref{eq:var2} to the function $F$ and so the same arguments gives 
  $$ \var_\mu(f)\le \int \sum_i c_i \big(\poinc{\mu_{x,E_i}}^{-1}-\alpha \big)^{-1} |P_{E_i} \nabla f(x)|^2 \, d\mu(x)
  +4\int \sum_i  d_i\poinc{\mu_{x,E_i}}|P_{E_i} \nabla f(x)|^2 \, d\mu(x).$$
 To conclude, note that, since $\alpha\ge 0$, we have $\poinc{\mu_{x,E_i}} \le \big(\poinc{\mu_{x,E_i}}^{-1}-\alpha \big)^{-1}$.
\end{proof}
\begin{remark}
Let us note that in all examples, the coefficients $c_i$ from the decomposition of the identity will indeed verify condition~\eqref{structcond}. If not, this property can however be enforced using  the invariance by conjugacy of $\mathcal G$, by a standard averaging procedure. Indeed, if $\sum c(R_i) P_{E_i} =\mathrm{Id}$, then conjugating by $g\in G$, 
$$
 \mathrm{Id} = g\cdot \left(\sum_i c(R_i) P_{E_i} \right) g^{-1}= \sum_i c(R_i) P_{gE_i}
    =\sum_j c(gR_jg^{-1}) P_{E_j},
$$
using that any $g\in G$ permutes the $m$-tuple $(E_1,\ldots,E_m)$. Averaging over $G$ yields
$$ \mathrm{Id}= \sum_j \left(\int_G c(gR_jg^{-1})\, d\gamma(g)\right) P_{E_j},$$
and the new coefficients $c'_j:=\int_G c(gR_jg^{-1})\, d\gamma(g)$ then verify~\eqref{structcond}.
% that $E_i=gE_j \Longrightarrow c'(E_i)=c'(E_j)$.
\end{remark}

A simple and useful instance of the previous theorem is  when $\poinc{\mathcal G}\le \kappa \min_i c_i$ (or more generally, when~\eqref{poincG} holds with $d_i\le \kappa c_i$). Indeed, we  can then conclude that the variance estimate \eqref{eq:var2} for invariant functions, passes to general function $f:\R^n\to \R$,  in the case when $-\alpha=\rho\le 0$,  with only an additional multiplicative factor:
 \begin{eqnarray}
  \var_\mu(f)&\le& (1+4\kappa) \int \sum_i c_i\, \big(\poinc{\mu_{x,E_i}}^{-1}-\alpha \big)^{-1} |P_{E_i} \nabla f(x)|^2 \, d\mu(x).
  \label{eq:var3} \\
 &\le & (1+4\kappa)\sup_{i,x}\big(\poinc{\mu_{x,E_i}}^{-1}-\alpha \big)^{-1} \int |\nabla f |^2 \, d\mu
  \label{eq:var3bis}
\end{eqnarray}

Let us present a few examples where  we are in such a situation.
\begin{itemize}
\item Unconditional symmetries: in this case $m=n$, and $R_i$ is the reflection of hyperplane $\{x;\, x_i=0\}$.
Obviously $E_i=\R e_i$ and $c_i=1$ provide the decomposition $\sum_i P_{E_i}=\mathrm{Id}$.
The group $G$ generated by these $n$ reflections is commutative and isomorphic to $\{-1,1\}^n$ (with coordinate-wise multiplication).
With our notation $\poinc{\mathcal G}=1/4$. Indeed, by the usual tensorisation property it is enough to deal with the 
case $n=1$, which is quite simple as for $f:\{-1,1\}\to \R$,  and for the uniform probability on the two-points space
$\var(f)=(f(0)-f(1))^2/4$, while the Dirichlet form is $\int (f(x)-f(-x))^2 d\gamma(x)=(f(0)-f(1))^2$.
In particular for any $n$, $\poinc{\mathcal G}=1/4 \min_i c_i$, which ensures~\eqref{eq:var3}
for general functions, at the expense of a additional factor $(1+4\kappa)=2$. 

\item The group of symmetries of  the regular simplex of $\R^n$, say with barycenter at 0, 
can be represented as the group of permutations on its $n+1$ vertices $u_1,\ldots,u_{n+1}$. It is generated by $n(n+1)/2$ reflections, which correspond
to transpositions of two vertices. Let us denote them $R_{i,j}$ for $1\le i<j\le n+1$. Then $E_{i,j}=\R(u_i-u_j)$.
The corresponding decomposition of the identity is  
$$ \frac{2}{n+1}\sum_{1\le i<j\le n+1} P_{E_{i,j}} = \mathrm{Id},$$
(the appendix explains why all the coefficients can be taken equal. Taking traces gives their common value).
Hence in this case $c_{i,j}=\frac{2}{n+1}$.
When the generating set is taken to be the set $\{R_{ij}\}$ of all transpositions (which is indeed stable by conjugacy), the spectral 
gap of the Cayley graph of the symmetric group was computed by Diaconis and Shahshahani \cite{DS}. Their result can be stated
in our notation as $\poinc{\mathcal G}= \frac{1}{2(n+1)}$ which is again equal to $\min c_{i,j}/4$. Thus~\eqref{eq:var3} holds with 
$(1+4\kappa)=2$.

\item The group of symmetries of the regular $k$-gon in the plane, has $2k$ elements: $k$ rotations and $k$ reflections.
The set $\mathcal G=\{R_1,\ldots,R_k\}$ of reflections generate the group and is stable by conjugacy. The corresponding decomposition
of the identity is $\mathrm{Id_{\R^2}}=\frac2k \sum_{i=1}^k P_{E_i}$ hence $c_i=2/k$ for all $i$. 
Next, the Cayley graph $\Gamma$ of $G$ for the generating set $\mathcal{G}$ is a complete bipartite graph on $2k$ elements
(direct and undirect isometries being the two  sets of vertices). 
With our notation $2\poinc{\mathcal G}$ is the inverse of the smallest positive eigenvalue of $kI-A$, where $A$ is the adjancency matrix
of $\Gamma$ and $I$ the identity of the same size. The spectrum of this adjacency matrix is easily computed.
One obtains $\poinc{\mathcal G}=1/(2k)=\min c_i/4$, as before.

\item The latter two examples are finite irreducible reflection groups. These groups have been classified by Coxeter.
It is tempting to believe that a similar inequality between the coefficients of the 
decomposition of the identity (given by Lemma~\ref{lem:reflection}) and the spectral gap  holds for all of them. Diaconis and Shahshahani have expressed the spectral gap 
for Cayley graphs of groups, with conjugacy stable generating sets, in terms of the characters of irreducible 
decompositions. Since the representation theory of reflection groups is well understood, it is in principle 
possible to settle this question.

\item The group of direct isometries of the regular $k$-gon is generated by the set $\mathcal G$ of its non-trivial 
rotations. For each such rotation $R_i$, $1\le i\le k-1$, $E_i=\R^2$, hence one can take $c_i=1/(k-1)$ in the decomposition 
of the identity. 
The corresponding Cayley graph is the complete graph on $k$ vertices. This leads to $\poinc{\mathcal G}=1/(2k)$.
This example is not interesting in itself, since each $E_i$ is equal to the whole space. It will be of interest though 
as a factor in a product group.
\end{itemize}

It will be useful to consider groups of isometries having a product structure. This is the case of  reflection groups, which factor as direct products of irreducible reflection groups.
 So let use assume that $\mathcal G=\{R_1,\ldots,R_m\}$ can be written as 
a disjoint union $\mathcal G_1\cup \ldots \cup \mathcal G_\ell$, such that the sets $\mathrm{Fix}(\mathcal G_i)^ \bot$,
$i=1,\ldots,\ell$ are orthogonal. This implies in particular the $G$ is the direct product of the corresponding groups
$G_1,\ldots,G_\ell$, which act on different blocks of an orthogonal decomposition of $\R^n$.
For convenience let us denote $\mathcal G_j=\{R_{j,1},\ldots,R_{j,m_j}\}$ and $E_{j,i}=\mathrm {Fix}(R_{j,i})^\bot$. Note that conjugacy will respect the product structure. Therefore, the Poincar\'e inequality~\eqref{poincG} holds (with some abuse of notation) with $d_{ij} = \poinc{\mathcal G_j}$ for all $i\le m_j$, which satisfy~\eqref{structcond}. By doing so, we get in~\eqref{eq:var3}  a result sharper than by using the Poincar\'e constant over the whole product, which is $\poinc{\mathcal G} =\max_j \poinc{\mathcal G_j}$.
In particular,  if we have a decomposition of the identity $\{c_{ij}, E_{ij}\}$ with
\begin{equation}\label{eq:compij}
 d_{ij}=\poinc{\mathcal G_j}\le \kappa \min_i c_{j,i}
 \end{equation}
then,  for all function $f:\R^n\to\R$ that is locally Lipschitz,
%One can easily prove the following version of Proposition~\ref{prop:var}:
%Let $f$ be a smooth function on $\R^ n$ and $F= \int f\circ g \, d\gamma(g)$.
%Then
%$$\var_\mu(f) \le \var_\mu(F)+4 \sum_j \poinc{G_j}\int \sum_i c_P(\mu_{x,E_{j,i}}) |P_{E_{j,i}}\nabla f(x)|^2 d\mu(x).$$
 %
$$ \var_\mu(f)\le (1+4\kappa) \int \sum_{ij} c_{ij}\, \big(\poinc{\mu_{x,E_{ij}}}^{-1}-\alpha \big)^{-1} |P_{E_{ij}} \nabla f(x)|^2 \, d\mu(x).$$
 %It can be proved either by applying iteratively Proposition~\ref{prop:var}, or by extending it to Dirichlet forms involving coefficients
%$d_i$ in front of the term $(f(gR_i)-f(g))^2$ with the property that $d_i=d_j$ when $R_i$ and $R_j$ are conjugates.
%For such a product group, the condition for extending the variance estimate \eqref{eq:var2} to general functions becomes: for all $j$, 
%$ \poinc{\mathcal G_j}\le \min_i c_{j,i}$.
In particular, if the decomposition of the identity is obtained by concatenating decompositions of the identity for each $G_j$ on the
subspace were it acts, having checked the condition on each factor implies it on the whole space.
This allows to get new examples from the ones listed above, leading to interesting spectral inequalities,  as discussed in the next section. We will consider in particular direct products were the factors are isometry groups 
of simplices in different dimensions, or groups of (direct or general) isometries of regular polytopes in planes.

\section{Spectral gap estimate for log-concave measures}
\label{sec:spectralgap}
\subsection{Using weighted Poincar\'e inequalities}

We show next how to pass from the weighted Poincar\'e type inequalities that we have established so far to spectral gap estimates, by using inequalities like~\eqref{eq:var3}, which is stronger than~\eqref{eq:var3bis}.
Let $\mu$ be a log-concave probability measure, invariant by isometries $R_1,\ldots,R_m$ satisfying Setting~\ref{setting}.
We also assume that $\sum c_i P_{E_i}=\mathrm{Id}$, were $c_i>0$ satisfy that $c_i=c_j$ when $R_i$ and $R_j$ are conjugates in the group $G$ that these
isometries generate (recall that $E_i$ is the orthogonal of $F_i:=\mathrm{Fix}(R_i)$).
By Theorem~\ref{th:vargeneral}, we know that for any locally Lipschitz function $f$,
$$ \var_\mu(f)\le \int \sum_i (c_i+ 4d_i) \poinc{\mu_{x,E_i}} |P_{E_i} \nabla f(x)|^2 \, d\mu(x).$$
The restricted measures are also log-concave, so as in the proof of Theorem~\ref{th:varnorm}, we may apply the bound~\eqref{KLS}.
%$\poinc{\nu} \le 4\int_E |y|^2d\nu(y)$ valid  for any log-concave probability
%measure $\nu$ on a Euclidean space $(E, |\cdot |)$. 
We obtain, with the  probabilistic notation where $X$ is a random vector with distribution
$\mu$,
$$\var_\mu(f) \le 4\sum_{i=1}^m (c_i+4d_i) \dE\left[ \dE\Big[ | P_{E_i} X|^2\,  \Big| P_{F_i} X \Big]\,  \big|P_{E_i}\nabla f( X) \big|^2 \right].$$
Setting $\kappa=\max_i d_i/c_i$ and using the decomposition of the identity gives
\begin{eqnarray*}
\var_\mu(f) &\le & 4(1+4\kappa) \dE\left[ \sum_{i=1}^m c_i \dE\Big[ | P_{E_i} X|^2\,  \Big| P_{F_i} X \Big]\,  \big|P_{E_i}\nabla f( X) \big|^2 \right]\\
&\le&  4(1+4\kappa) \dE\left[ \max_{1\le i\le m} \left( \dE\Big[ | P_{E_i} X|^2\,  \Big| P_{F_i} X \Big]\right)\,  \big|\nabla f( X) \big|^2 \right].
\end{eqnarray*}
The weight in front of the gradient may be unbounded.
However a deep result of E. Milman  \cite{E.M} ensures that for log-concave probabilities, any weight appearing in front of the gradient in such inequality can
be "averaged out" of the integral up to a numerical constant (because of the equivalence between $L^2$ and $L^\infty$ Poincar\'e inequalities), and so we get that  there is a universal constant $c_0$ such that
\begin{equation}\label{eq:max}
 \poinc{\mu}\le c_0(1+4\kappa) \dE \left[ \max_{1\le i\le m} \dE\Big[ | P_{E_i} X|^2\,  \Big| P_{F_i} X \Big] \right].
\end{equation}
To estimate the latter expectation, we use the so-called $\psi_1$-property of log-concave distribution (which is related to the result of 
Borell~\cite{Borell} recalled at the end of Section~\S\ref{sec:L2}): there exists a universal constant $c$ such 
that  for any $k\ge 1$, any log-concave random vector $Y$ in $\R^k$ and any even  1-homogeneous convex function $H:\R^k\to \R^+$, 
$$\dE \exp\Big( \frac{H(Y)}{c\sqrt{\dE\big( H(Y)^2\big)}}\Big)\le 2.$$
Setting $D=c^2 \max_i \dE \big( | P_{E_i} X|^2\big)$, using the convexity of $\psi_{1/2}(t)=\exp(\sqrt{1+t})$, $t\ge0$ and applying Jensen inequality twice, we get
\begin{eqnarray*}
\psi_{1/2} \left( \dE \left[ \max_{1\le i\le m} \dE\Big[ D^{-1}| P_{E_i} X|^2\,  \Big| P_{F_i} X \Big]\right] \right)
&\le &  \dE \max_{1\le i\le m} \dE \Big[ \psi_{1/2} \left(  D^{-1}| P_{E_i} X|^2 \right)\,  \Big| P_{F_i} X \Big] \\
\le   \dE \sum_{1\le i\le m} \dE \Big[ \psi_{1/2} \left(  D^{-1}| P_{E_i} X|^2 \right)\,  \Big| P_{F_i} X \Big] &=&
\sum_{1\le i\le m} \dE \left[ \psi_{1/2} \left(  D^{-1}| P_{E_i} X|^2 \right)\right]\\
\le e   \sum_{1\le i\le m} \dE \left[ \exp \left(  D^{-1/2}| P_{E_i} X| \right) \right] &\le & 2e m.
\end{eqnarray*}
%We have also used that the projection of a log-concave random vector is also log-concave.
This estimates leads to 
$$\poinc{\mu}\le c' (1+4\kappa) (1+\log m)^2 \max_i \dE \big( | P_{E_i} X|^2\big).$$
If $\mu$ is also isotropic, then $\poinc{\mu}\le c' (1+4\kappa) (1+\log m)^2 \max_i \mathrm{dim}(E_i)$.
This estimates depends on the number of isometries involved, so it is better to choose them parsimoniously. In particular, note that
in the previous estimation of $ \dE \left[ \max_{1\le i\le m} \dE\Big[ | P_{E_i} X|^2\,  \big| P_{F_i} X \Big] \right]$, what really matters is  the number of different terms,
\begin{equation}
\label{def:m'}
m':={\rm Card}\big(\{ E_i \; ; \  i=1,\ldots, m\}\big),
\end{equation}
for if $E_i=E_j$ then $\dE\Big[ | P_{E_i} X|^2\,  \big| P_{F_i} X \Big] = \dE\Big[ | P_{E_j} X|^2\,  \big| P_{F_j} X \Big] $ which contributes only once in the supremum. So under the same assumptions, the bound  for a log-concave isotropic $\mu$ is rather
\begin{equation}
\label{eq:poincsym}
\poinc{\mu}\le c' (1+4\kappa) (1+\log m')^2 \max_i \mathrm{dim}(E_i).
\end{equation}
Let us illustrate the interest of such bound in  the following application, which extends Klartag's result to more general 
sets of reflections.
\begin{corollary}
Let $H_1,\ldots,H_m$ be hyperplanes of $\R^n$ such that $\cap_{i=1}^m H_i=\{0\}$.
Let $\mu$ be an isotropic log-concave measure on $\R^n$, which is invariant by the hyperplane symmetries $S_{H_1},\ldots,S_{H_m}$.
Then $ \poinc{\mu}\le c \log(1+n)^2,$
where $c>0$ is a numerical constant.
\end{corollary}
\begin{proof}
The group generated by the reflections $G$ that leave $\mu$ invariant might be quite large and we want to apply~\eqref{eq:poincsym} to a well chosen set of generators as discussed at the end of the previous section. Indeed,  the structure
of reflection groups is fully described by Coxeter's classification theorem. As explained in the appendix, there is an orthogonal 
decomposition $\R^n=\mathcal E_1 \oplus \cdots \oplus \mathcal E_\ell$, such that $G$ contains a subgroup $G'=\mathcal O(P_1)\times \cdots
\times \mathcal O(P_\ell)$ where $P_i$  is either a regular simplex of full dimension in $\mathcal E_i$ or a regular polygon
 if $\mathcal E_i$ is a plane. Set $G_i=\mathcal O(P_i)$ if $P_i$ is a regular simplex in dimension at least 3 and $G_i= \mathcal {SO}(P_i)$ if $P_i$
 is a regular polygon, and set $G'':=G_1\times \cdots\times G_\ell$. This subgroup of $G$ only fixes the origin.
 Let us describe precisely the generators of $G''$ that we consider:
\begin{itemize}
\item If $n_i:=\mathrm{dim}(\mathcal E_i)\ge 3$, then $\mathcal O(P_i)$ is generated by the $n_i(n_i+1)/2$ hyperplane symmetries of  a regular simplex in $\mathcal E_i$.
  We denote
 by $R_{i,j}$, $1\le j\le n_i(n_i+1)/2$ the hyperplane reflections on $\R^n$ which extend them to $\R^n$ (by acting as the identity on $\mathcal E_i^\bot$).
\item   If $n_i:=\mathrm{dim}(\mathcal E_i)=2$, then $P_i$ is a regular $k_i$-gon. We consider all its non-trivial rotations $R_{i,j}$, $j=1,\ldots,k_i-1$ as acting on 
 the whole $\R^n$. Obviously $E_{i,j}=\mathrm{Fix}(R_{i,j})^\bot =\mathcal E_i$ is of dimension 2 
 and $F_{i,j}=\mathrm{Fix}(R_{i,j})=\mathcal E_i^\bot$. These spaces do not
  depend on $j$, and so the contribute only once in~\eqref{def:m'}.
%  $$ \max_{1\le j\le k_i-1} \dE\Big[ | P_{E_{i,j}} X|^2\,  \big| P_{F_{i,j}} X \Big]=   \dE\Big[ | P_{E_{i,1}} X|^2\,  \big| P_{F_{i,1}} X \Big].$$
 \end{itemize} 
We define $\mathcal G''$ as the set of all $R_{i,j}$'s. It generates $G''$ and is stable by conjugacy. Moreover, by the remarks at the end of the previous section,
 $G'',\mathcal G''$ satisfy all the required hypotheses to carry on the previous analysis of $\poinc{\mu}$ (a decomposition of the identity
is obtained by putting together the ones on individual $\mathcal E_i$'s. Note that it is obvious when $n_i=2$. Also $\kappa=1/2$).
In order to conclude, we just need to estimate the number $m'$  (which is smaller than the number of generators in $\mathcal G''$):
$$ m'=\sum_{i, n_i\neq 2} \frac{n_i(n_i+1)}{2} + \mathrm{card}\big(\{i;\; n_i=2\}\big)\le 
   \sum_{i, n_i\neq 2} n_i^2 + \sum_{i;\; n_i=2} n_i^2 \le \big(\sum_i n_i\big)^2=n^2.$$
   Consequently, from~\eqref{eq:poincsym} we get that $\poinc{\mu}\le c\log(1+n)^2$. 
 \end{proof}
 
 The method also applies to the case of the Schatten classes. The notation here are those from the end of Section~\S\ref{varnorm}.
So  $X_{p,n}\sim \mu_{\lambda B_p^d}$ is an isotropic log-concave random vector uniformly distributed on a multiple $\lambda B_p^d$ of the unit ball of the Shatten space $S_{p}^d$. The dimension is $n=d^2$. Consider again the  linear applications $R_{i}$ which flip the signs of all the entries in the $i$-th row of a matrix, $i=1,\ldots, d$. The orthogonal of the fixed point subpaces decompose the identity (with constant coefficients), and the group generated by the $R_i$ is $\{-1,1\}^d$ (because the $R_i$'s commute and $R_i^2=\id$), which has also constant spectral gap. So the estimate~\eqref{eq:poincsym} applies and we get that
$$\poinc{X_{n,p}} \le c \sqrt n \, \log(1+n) ^2,$$
for some universal constant $c>0$.

\subsection{Anti-invariance of eigenfunctions}

The previous subsection recovers Klartag's bound on the spectral gap for isotropic unconditional bodies \cite{KlartagPTRF}, but by a different method.
Below we briefly present a streamlined version of his argument, which applies to rather general symmetries (even in comparison
to the previous subsection).

We consider a (strictly) log-concave probability measure $\mu$ with density $e^{-\Phi}$ with $\Phi$ smooth and $D^2 \Phi >0$. 
We also need that the spectral gap be achieved, i.e that there is  a (smooth)  $\varphi \in L^2(\mu)$ belonging to the domain of $L$ such that
$$L\varphi = - \lambda_1(\mu) \varphi, \qquad \text{where} \ \ \lambda_1(\mu)  := \frac1{\poinc{\mu}}.$$
This is the case if, for instance, $D^2 \Phi (x) \ge \varepsilon_0 \id$ for all $x\in \R^n$, for some $\varepsilon_0>0$.
Note that this condition may be ensured by adding  $\varepsilon_0 |x|^2/2$ to the potential $\Phi(x)$ for an arbitrarily small  $\varepsilon_0>0$.

\begin{proposition} Let $\mu$ be a strictly log-concave probability measure. Let  $R_1, \ldots ,R_m \in \O_n(\mu)$   such that  $\bigcap_{i\le m} \fix(R_i) = \{0\}$. 
Then, if $u$ is a non-zero eigenfunction for $\lambda_1(\mu)$, there exists  $j\le m$ such that the function
\begin{equation}\label{fonctionsym}
\varphi(x):=u(R_j x) - u(x), \qquad \forall x \in \R^n
\end{equation}
is again a non-zero eigenfunction for $\lambda_1(\mu)$. 
%This eigenfunction has the property that for every  $x\in \R^n$, setting $E_j= \fix(R_j)^\perp$, we have
%$$\int_{E_j} \varphi\, d\mu_{x,E_j}  \; = \; 0.$$
\end{proposition} 

It is worth noting that when the $R_i$'s are such that $R_i^2=\mathrm{Id}$ (for instance when the $R_i$ are reflections or $-\id$), 
then the function $\varphi$ of the theorem verifies $\varphi(R_j x)= -\varphi(x)$ for all  $x\in \R^n$.
In the case of even measures and of unconditional measures, the existence of an eigenfunction verifying this anti-symmetry was proved by Klartag~\cite{KlartagPTRF}.

\begin{proof}
Let $u$ be an eigenfunction. Note that the functions $u\circ R_i$ are also eigenfunction, so if the first conclusion of the theorem is false 
(i.e.  if for any choice of $j$ the corresponding  $\varphi$ is zero), it means that for all $i\le m$ we have that $u\circ R_i = u$. 
Taking gradients, integrating with respect to $\mu$ and using its invariance by $R_i$ then gives $R_i^{-1} \int \nabla u \, d\mu=\int \nabla u \, d\mu$.
Hence $\int \nabla u \, d\mu \in \cap_i \mathrm{Fix}(R_i)$. 
Our assumptions on the $R_i$'s then implies that 
$$\int \nabla u \, d\mu = 0 .$$
Consequently,
$\int \partial_i u \, d\mu =0$ for $i=1, \ldots , n$, and therefore by the Poincar\'e inequality and by~\eqref{bochner}, we have  
\begin{gather*}
\lambda_1(\mu) \int u^2 \, d\mu = \int |\nabla u|^2\, d\mu \le \frac1{\lambda_1(\mu)} \sum_{i=1}^n  \int |\nabla \partial_i u |^2\, d\mu =  \frac1{\lambda_1(\mu)}  \int \|D^2 u \|^2\, d\mu \\
\le  \frac1{\lambda_1(\mu)}  \int (Lu)^2 \, d\mu = \lambda_1(\mu) \int u^2 \, d\mu.
\end{gather*}
Equality means, from the use of~\eqref{bochner},  that    
$D^2 \Phi (\nabla u) \cdot \nabla u = 0$ on  $\R^n$ and so $\nabla u = 0$. Therefore  $u$ is constant, and this constant must be zero ($\lambda_1>0$). 
This shows that when we start from a non-zero eigenfunction $u$, we can indeed find a $j\le m$ such that the corresponding $\varphi= u\circ R_j - u$ is a non-zero eigenfunction.  

\end{proof}

The interest of having a non-zero eigenfunction of the form~\eqref{fonctionsym} comes from the  mean zero property given by Fact~\ref{meanzero}.
Using the same argument as Klartag~\cite{KlartagPTRF}, we can deduce from the existence of an ``anti-invariant" eigenfunction an estimation on the spectral gap, similar to the bound~\eqref{eq:var3bis} holding under different hypotheses.

\begin{corollary}\label{poinclogconcave1}
 Let $\mu$ be a log-concave function on $\R^n$ and  $R_1, \ldots ,R_m \in \O_n(\mu)$ such that  $\bigcap_{i\le m} \fix(R_i) = \{0\}$.
 Assumed that the spectral gap is achieved. Then we have, setting $E_i= \fix(R_i)^\perp$,
 $$\poinc{\mu} \le \max_{i\le m}\sup_{x\in \fix(R_i)}  \poinc{\mu_{x,E_i}}.$$
\end{corollary}

\begin{proof}
By the previous Proposition and Fact~\ref{meanzero},
 we know that there exists $j\le m$ and a non-zero eigenfunction such that $\int_{E_j} \varphi\, d\mu_{x,E_j}  \; = \; 0$ for all $x\in \R^n$. This implies that 
$$\int_{E_j} \varphi^2 \, d\mu_{x,E_j}  \le \poinc{\mu_{x,E_j}} \int_{E_j} |P_{E_j}\nabla \varphi|^2 \, d\mu_{x,E_j} \le C(\mu) \int_{E_j} |\nabla \varphi|^2 \, d\mu_{x,E_j} ,$$
setting $C(\mu) := \max_{i\le m}\sup_{x\in \fix(R_i)}  \poinc{\mu_{x,E_i}}$. The conclusion follows from Fubini's theorem, since by definition $\int|\nabla \varphi |^2 \, d\mu = \poinc{\mu}^{-1} \int \varphi^2 \, d\mu$.
\end{proof}

To follow Klartag's proof in order to reach bounds depending only on $n$ for isotropic measures, one  upper bounds $\poinc{\mu_{x,E_i}}$ as before,
considers the probability $\mu_{|A}$ measure obtained by restricting $\mu$ to a set $A=\{x\in\R^n; \forall i\le m,\, |P_{E_i}x|\le C\},$
where $C$ should be tuned to ensure  $\mu(A)\ge 1/2$ and thus $d_{TV}(\mu,\mu_{|A})\le 1$ (while the maximum distance is 2).
Another result by E. Milman \cite{E.M} ensures  $\poinc{\mu}\approx \poinc{\mu_{|A}}$. If $\mu_{|A}$ has the same invariances (which requires
stability by conjugacy of the set of isometries), then applying the previous theorem to $\mu_{|A}$ gives $\poinc{\mu_{|A}}\le c' C^2$.
Eventually, one chooses $C$ to be an upper estimate of $2 \dE \max_i |P_{E_i}X|$ which is derived from the $\psi_1$ property of 
log-concave distributions. We skip the details.

\section{An application to conservative spin systems}
\label{sec:spin}
We now apply the previous tools to a conservative non-interacting unbounded spin system.
Below, $\mu$ will be a probability measure on $\R$ of the form $d\mu(t)=e^{-V(t)} dt$. For $n\ge 2$ and $m\in \R$, we consider
the probability measure obtained by restriction of the product measure $\mu^n=\mu\otimes\ldots\otimes\mu$ on $\R^n$ 
to the affine hyperplane $H^n_m=\{x\in \R^n;\; \sum_i x_i=nm\}$:
% inherited from $\R^ n$):
$$\mu^{n|m}:= \mu^n\Big( \;\cdot\;\big| \sum_{i=1}^n x_i=nm\Big) .$$
Equivalently, the density of $\mu^{n|m}$ with respect to Lebesgue's measure of $H^n_m$ is proportional
to $$\exp\left(-\sum_{i=1}^n V(x_i)\right).$$
As always, $H^n_m\simeq H^n_0\simeq \R^{n-1}$ is viewed as an Euclidean space for the structure inherited from $\R^n$.

Ergodic inequalities have been studied for these measures. Varadhan posed the following question \cite{V}: for which
kind of single-site potentials $V$ is it true that the measures $\mu^{n|m}$ have a uniform spectral gap, i.e.
$\sup_{n,m} \poinc{\mu^{n|m}}<+\infty$? The same question may be also asked for the stronger logarithmic Sobolev 
inequalities. After contributions by  Landim-Panizo-Yau \cite{LPY}, Caputo \cite{C}, Chafai \cite{Ch}, Grunewald-Otto-Villani-Westdickenberg \cite{GOVW},
a very satisfactory answer was recently given by Menz and Otto \cite{MO}. They show that if $V=\phi+\psi$, with
$\inf_x\varphi''(x)>0$ and $\|\psi\|_\infty, \|\psi'\|_\infty<+\infty$, then the measures $\mu^{n|m}$ satisfy a
log-sobolev inequality with constants which do not depend on $n,m$.
In particular, if $V$ is a bounded and Lipschitz perturbation of a strictly uniformly convex function, then 
$\sup_{n,m} \poinc{\mu^{n|m}}<+\infty$.

The arguments of these articles are quite involved. They rely heavily on the fact that the spectral gap 
is uniform, which allows induction on coordinates or coarse-graining approaches. Our goal is to show
how the symmetries can be exploited to provide a soft proof of various results for these measures.
Some of them apply to the sub-quadratic case where it is known that the spectral gap cannot be uniform
in $m$. Not much is known in this case (note that the first named author and Wolff \cite{BW} have obtained
precise estimates of the spectral gap and log-Sobolev constants when $\mu$ is a Gamma distribution with parameter
at least 1).

The measure $\mu^n$ is obviously invariant by permutations of coordinates. Since $H^n_m=\{x\in \R^n;\; \sum_i x_i=nm\}$
is invariant as well, it follows that $\mu^{n|m}$ is invariant by the restrictions to $H^n_m$ of permutations of coordinates
of $\R^n$. They obviously act as a permutation of the points $nm e_1,\ldots, nm e_m$ which form a regular simplex in $H_m^n$
with barycenter at  $(m,\ldots,m)$. It is convenient to consider this point as the new origin for this hyperplane.
Summarizing, $\mu^{n|m}$ is invariant by the group of isometries of a regular simplex.  Our previous results apply, and require
the estimation of the Poincar\'e constants of restrictions to lines which are orthogonal to the hyperplanes of symmetries. These restrictions
have a very simple structure, as we put forward next.

For $1\le i<j\le n$, let $S_{i,j}$ the hyperplane symmetry defined on $\R^n$ by $S_{i,j}(e_k)=e_{\tau_{i,j}(k)}$ where $\tau_{i,j}$
is the transposition of $i$ and $j$, i.e. $S_{i,j}=S_{(e_i-e_j)^\perp}$. We shall view  $S_{i,j}$ as a hyperplane symmetry of $H^n_m$. Setting $E_{i,j}:=\mathrm{Fix}(S_{i,j})^\bot= \R (e_i-e_j)$, we can note that for $x\in H^n_m$ we have $x+E_{i,j}\subset H_m^n$.

\begin{lemma}\label{lem:mu2}
Let $n\ge 2$. For any $x\in H^n_m$ and all $i<j$,
$$\poinc{\mu^{n|m}_{x,E_{i,j}}}=\poinc{\mu^{2\, |(x_i+x_j)/2}}.$$
\end{lemma}
\begin{proof}
Since permutations of coordinates leave $\mu^{n|m}$ invariant, we can  assume that $i=1$, $j=2$.
Then
\begin{eqnarray*}
x+E_{1,2}&=& \left\{\Big(x_1+\frac{s}{\sqrt2},x_2-\frac{s}{\sqrt2},x_3,\ldots,x_n\Big);\; s\in \R \right\}\\
&=& \left\{\Big(\frac{x_1+x_2}{2}+\frac{t}{\sqrt2},\frac{x_1+x_2}{2}-\frac{t}{\sqrt2},x_3,\ldots,x_n\Big);\; t\in \R \right\},
\end{eqnarray*}
where we have chosen parametrizations of unit speed.
Hence the density of $\mu^{n|m}$ at point of $x+E_{1,2}$ is proportional to 
$$ \exp\left(-V\Big(\frac{x_1+x_2}{2}+\frac{t}{\sqrt2} \Big)-V\Big(\frac{x_1+x_2}{2}-\frac{t}{\sqrt2} \Big)-V(x_3)-\cdots-V(x_n)\right).$$
Note that the last $n-2$ terms are constant on $x+E_{1,2}$. Hence they disappear when conditioning  $\mu^{n|m}$ to this line: 
$$\mu^{n|m}_{x,E_{i,j}}(dt)= \frac1Z  \exp\left(-V\Big(\frac{x_1+x_2}{2}+\frac{t}{\sqrt2} \Big)-V\Big(\frac{x_1+x_2}{2}-\frac{t}{\sqrt2} \Big)\right) \, dt,$$
where $Z$ is the normalization constant. 
This measure corresponds to the normalized restriction of $\mu^2$ to the line $\{y\in \R^2;\; y_1+y_2=x_1+x_2\}$, that is to $\mu^{2|(x_1+x_2)/2}.$
\end{proof}

The next statement is then an obvious consequence of the results of the Section~\ref{sec:symmetrisation}.
\begin{proposition}\label{prop:2n}
Assume that there exists $\alpha\in \R^+$ such that for all $t\in \R$, $V''(t)\ge -\alpha$ and $\sup_m \poinc{\mu^{2|m}}<\frac1\alpha$.
Then, 
$$ \sup_{n,m} \poinc{\mu^{n|m}} \le 2 \sup_m \left(\poinc{\mu^{2|m}}^{-1}-\alpha \right)^{-1}.$$
\end{proposition}

\begin{proof}
Recall that for a measure $d\nu(x)=e^{-W(x)}dx/Z$, it is convenient to call $W$ a potential of $\nu$.
Since $\mu$ has a potential with second derivative bounded from below by $-\alpha$,
so does the product measure $\mu^n$: the Hessian of its potential is bounded from below by $-\alpha \mathrm{Id}$.
By restriction, this property passes to $\mu^n_m$. Since the latter measure has the symmetries of the 
regular simplex, we are in a position to apply Theorem \ref{th:vargeneral} by the remarks after~\eqref{eq:var3bis} in the case of invariances by the symmetric group. Lemma \ref{lem:mu2} allows us to deal with the restrictions to lines which are orthogonal to hyperplanes
of symmetries.
\end{proof}

In the sequel, let us adopt the notation $a\approx b$ for the existence of numerical constants $c,C>0$ such that $ca \le b \le c B$; by numerical constants we mean  universal computable constants (and so independent of $n$, $m$, $V$, etc.)

In the case convex single-site potentials, we can get a precise quantitative estimate of the 
uniform Poincar\'e constant (finite or infinite):

\begin{proposition} 
If the potential $V$ is convex, then
$$ \sup_{n,m} \poinc{\mu^{n|m}} \approx \sup_{m} \poinc{\mu^{2|m}}\approx
 \sup_{m\in\R} \left( \int_{0}^{+\infty} e^{-\big[V(m+t)+V(m-t)-2V(m)\big]}dt\right)^2.$$
 \end{proposition} 
 \begin{proof}
 Applying the previous Proposition with $\alpha=0$ (recall the convention for non-negative numbers
 $1/0=+\infty$) gives
 $$  \sup_{m} \poinc{\mu^{2|m}}\le  \sup_{n,m} \poinc{\mu^{n|m}}\le 2 \sup_{m} \poinc{\mu^{2|m}},  $$
 which gives the first approximate equality. Note that applying Corollary~\ref{poinclogconcave1} would remove the
 factor 2 and give an equality (this may require an approximation argument in order to ensure the existence
 of an eigenfunction corresponding to the spectral gap).

Note that $\mu^{2|m}$ can be viewed as the probability measure on $\R$ with density $f(t) dt/ \int f$, where
$$ f(t)=\exp\left( -V\Big(m+\frac{t}{\sqrt 2} \Big)+ V\Big(m-\frac{t}{\sqrt 2} \Big)\right)$$
is an even log-concave function on $\R$. A result of Bobkov~\cite{Bobkov:1999vk} (see also~\cite{KLS}) states that for log-concave measures on the line, the bound~\eqref{KLS} is sharp, and so
$$\poinc{\mu^{2|m}}\approx \var_{\mu^{2|m}}(t)= \frac{\int_\R t^2 f(t) dt}{\int_\R f(t) \,dt}.$$
Next, we apply a very classical fact about even log-concave measures on the real line (see e.g.  \cite{Ball,MP}):
$$ \frac13 \le \frac{f(0)^2 \int_{0}^{+\infty} t^2 f(t)\, dt}{\left( \int_0^{+\infty} f(t)\, dt\right)^3} \le 2.$$  
It follows that 
$$\poinc{\mu^{2|m}}\approx\left( \frac{\int_0^{+\infty} f}{f(0)}\right)^2=\left(\int_0^{+\infty} \exp\left(-V(m+t/\sqrt2)-V(m-t/\sqrt2)+2V(m)  \right) \, dt\right)^2.$$
The claim follows after an obvious change of variables.
 \end{proof}
 
 The previous characterization allows to distinguish two different types of behaviours:
 
 \begin{itemize}
 \item If $V(t)=|t|$, one gets that $\int_{0}^{+\infty} e^{-\big[V(m+t)+V(m-t)-2V(m)\big]}dt=|m|+1/2$ and there is no uniform spectral gap. The same happens if $V(t)=|t|^\beta$ for $\beta\in [1,2)$ as a study of the corresponding integrals for $m\to \infty$ shows.

\item  If any strict uniform convexity property of $V$ holds, of the form
$$\forall m, t \in \R, \ \ V(m+t)+V(m-t)-2V(m) \ge \omega(t) \quad \textrm{with}\quad \int_\R e^{-\omega}<\infty,$$
then we can guarantee a uniform spectral gap. This condition is more general than the usual strict uniform convexity $\inf V''>0$ and is verified e.g. by potential $V(t)=t^\beta$, $\beta>2$ (see below). 
\end{itemize}

 We conclude this section with an example of application. It is definitely
less encompassing than the recent Menz-Otto theorem, but covers some cases of non-convex potentials which were not treated by Caputo for
technical reasons.

\begin{corollary}
Let $C:\R^+\to \R$ be a convex non-decreasing function. Let $\psi:\R^+\to \R$ be twice continuously differentiable
with $\|\psi\|_\infty, \|\psi''\|_\infty<+\infty$. Let $V_\varepsilon(t)=C(t^2)+\varepsilon \psi(t)$ and $\dmu_\varepsilon(t)=e^{-V_\varepsilon(t)} dt/Z_\varepsilon$ the corresponding
probability measure. Then there exists $\varepsilon_0>0$ such that for all $\varepsilon \in [-\varepsilon_0,\varepsilon_0]$,
 $$\sup_{n,m} \poinc{\mu^{n|m}_\varepsilon}<+\infty.$$
\end{corollary}
\begin{proof}
Let us start with the case $\varepsilon=0$. The single site potential is convex and we may apply the previous proposition.
Thanks to the convexity of $C$, used twice
\begin{eqnarray*}
V_0(m+t)+V_0(m-t)-2V_0(m) &=& C(m^2+t^2+2mt)+ C(m^2+t^2-2mt)-2C(m^2)\\
 &\ge & 2 \big(C(m^2+t^2)-C(m^2) \big) \ge  2 \big(C(t^2)-C(0)\big)\\
 &=&2(V_0(t)-V_0(0)) \ge V_0(t)-V_0(0).
\end{eqnarray*}
Hence 
$$\int_{\R^+} \exp(-[V_0(m+t)+V_0(m-t)-2V_0(m)])dt \le \int \exp(-V_0(t)+V_0(0))dt=Z_0e^{V_0(0)}<+\infty.$$
The case of $\varepsilon\neq 0$ is obtained by perturbation. First note that $V_\varepsilon''(t)\ge -|\varepsilon|  \,\|\psi''\|_\infty$.
Setting $\alpha:=|\varepsilon|  \,\|\psi''\|_\infty$, we deduce that the potential $\Phi$ of $\mu_\varepsilon^{n|m}$ satisfies
$D^2\Phi\ge -\alpha \mathrm{Id}$ pointwise. 

Thanks to Proposition~\ref{prop:2n}, proving that for all $m$, $\poinc{\mu_\varepsilon^{2|m}}\le 1/(2\alpha)$
would be enough to establish the uniform spectral gap inequalities for the measures $\mu^{n|m}_\varepsilon$, $n\ge 2, m\in \R$.
But this follows from the classical Holley-Stroock bounded perturbation principle (on $\R$): if $d\nu=e^g d\mu$ then $\poinc{\nu}\le e^{\sup g-\inf g} \poinc{\mu}$.
Indeed the potential of $\mu_\varepsilon^{2|m}$ differs from the one of $\mu_{0}^{2|m}$ only by the term $\varepsilon \psi(m+t/\sqrt 2)+ \varepsilon \psi(m-t/\sqrt 2)$
(and a constant term coming from normalization, which does not contribute to the oscillation of the perturbation).
Hence 
$$\poinc{\mu_\varepsilon^{2|m}}\le e^{4 |\varepsilon| \|\psi\|_\infty} \poinc{\mu_0^{2|m}} \le c (Z_0 e^{V_0(0)})^2 e^{4 |\varepsilon| \|\psi\|_\infty},$$
where we have used the estimate  established in the $\varepsilon=0$ case.
Consequently, if $|\varepsilon|$ verifies that
$$ c (Z_0 e^{V_0(0)})^2 e^{4 |\varepsilon| \|\psi\|_\infty} \le \frac{1}{2|\varepsilon|  \,\|\psi"\|_\infty},$$
then we have uniform Poincar\'e constant for the measures $\mu^{n|m}_\varepsilon$, $n\ge 2, m\in \R$. This is obviously true 
when $\varepsilon$ is close enough to zero.
\end{proof}

Finally, let us state a result for log-concave single-site potentials, which just uses  the symmetries of $\mu^{n|m}$.
It follows from what we already proved for log-concave measures with the symmetries of the simplex:
\begin{theorem}
Let $\mu$ be a log-concave measure on $\R$. Let $\ell:\R^n \to \R$ defined by $\ell(x)=x_i\sqrt{n/(n-1)}$, for any $i\le n$.
Then for all $n\ge 2$, $m\in \R$,
$$\var_{\mu^{n|m}}(\ell) \le \poinc{\mu^{n|m}} \le c (\log n)^2\,  \var_{\mu^{n|m}}(\ell),$$
where $c$ is a universal constant.
\end{theorem}
\begin{proof}
Since the measure $\mu^{n|m}$ is invariant by an irreducible groups of isometries, it is automatically a dilate (by $\sqrt{\var_{\mu^{n|m}}(\ell)}$) of an isotropic distribution.
Note that $\ell$ is a linear function with unit length gradient (in  $H_m^n$ Euclidean structure).
Se the left-hand side inequality just follows by applying the Poincar\'e inequality to $\ell$. The right-hand side
inequality is a particular case of what we have proved for isotropic log-concave measures having the symmetries of the simplex.
\end{proof}
The previous result applies even when there are no uniform bounds. The  structure of the measure may be used
to estimate precisely the variance of linear functions (say of $\ell$).

%%%%%%%%%%%%%%%%%%%%%%%%%%%%%%%%%%%%%%%%%%%%%%%%%%%%%%%%%%%
%%%%%%%%%%%%%%%%%%%%%%%%%%%%%%%%%%%%%%%%%%%%%%%%%%%%%%%%%%%

\section{Isotropy constant of bodies with invariances}
\label{sec:isotropy}

This section investigates bounds on the isotropic constant of convex bodies (having, as before, many invariances). This  problem is central in the asymptotic theory of convex bodies, and is closely related to the questions discussed in previous sections, although the methods we will use here are rather different. 
We will work with convex bodies rather than with measures, mainly for convenience and for historical reasons. 

Recall that the isotropy constant of a convex body $K\subset \R^n$ is the positive number defined by
$$ L_K^2=\inf_{T\in A(\R^n)} \frac{1}{|TK|^{1+\frac2n}} \int_{TK} \frac{|x|^2}{n} \, dx,$$
where $A(K)$ denotes the affine group. The infimum is achieved for a map $T_0$ if and only if 
the barycenter of $T_0K$ is at the origin, and there exists a constant $M$ such that for all $\theta\in \R^n$,
$$\frac{1}{|T_0K|^{1+\frac2n}} \int_{T_0K} \langle x,\theta\rangle^2 \, dx=M |\theta|^2.$$
One then says that $T_0K$ is in \emph{isotropic position}.
Note that necessarily, $M=L^2_K$ and that $K\to \int_K |x|^2dx/|K|^{1+2/n}$ being invariant by dilations,
one may find a minimizer also satisfying $|T_0K|=1$. In the sequel, we shall also assume that our convex bodies have barycenter at the origin.

A major open problem is whether the numbers $L_K$ are uniformly bounded, independently of the dimension (a classical reference is \cite{MP}). This question is in fact related to the variance conjecture, as established in \cite{EK}.
It will be convenient to define $L(d)$ as the supremum of the isotropy constant
of  convex bodies in $\R^d$. The best known upper bound is due to Klartag \cite{K}: $L(d)\le c d^{1/4}$.

Given a subspace $E\subset \R^n$, the measures $|F\cap E|$ and $|P_E(K)|$ refer to the Lebesgue measure $|\cdot|$ in Euclidean space $E$; if $d=\textrm{dim}(E)$, we shall sometimes use also the notation $|\cdot|_d$, for clarity. By convention we have  $|A|_d=1$ if $d=0$ (i.e. $E=\{0\}$) and $0\in A$.

The isotropy constant is related to the size of sections of bodies in isotropic position. This principle
goes back to Hensley \cite{H}. The next statement appears in the lecture notes by Giannopoulos~\cite[pp 60-61]{G};
it is a non-symmetric version of a result of  Ball~\cite{Ball} (see also \cite{F} for sharp constants in the
hyperplane case).

\begin{theorem}\label{th:sectionB}
   Let $n>m$ and let $C\subset \R^n$ be a convex body in isotropic position.
  Let $E$ be a subspace 
of $\R^n$ with codimension $m$. Then
   $$ |C\cap E|^{\frac1m} L_C
 \le \kappa \, L(m) \, |C|^{\frac1m-\frac1n},$$
 where $\kappa$ is a universal constant.
\end{theorem}

Let us emphasize a useful property of subspaces obtained as fixed-point spaces of an isometry of $K$.
\begin{lemma}\label{lem:fix}
Let $K\subset \R^n$ be a convex body and  $U\in \mathcal O_n(K)$.Then, for  $F:=\mathrm{Fix}(U)$ we have
$$ P_{F}K=K\cap F.$$
\end{lemma}
\begin{proof}
For $k\ge 1$, let $U_k:=(\mathrm{Id}+U+\cdots+U^{k})/(k+1).$ We use that $\lim_{k\to \infty} U_k=P_F$
(to see this, simply diagonalise $U$ over $\mathbb C$).
%: the eigenspace for 1 is   $F$, and all the other eigenvalues are in $\{z\in \mathbb C\setminus\{1\};\; |z|=1\}$). 
The convexity of $K$ ensures that $U_kx\in K$ for every $x\in K$,
and taking limits gives $P_Fx\in K$. Hence $P_F K\subset K\cap F$.  
\end{proof}

  As noted by several authors in the eighties, unconditional convex bodies have a bounded isotropy constant. The next statement
gives a similar result for more general symmetries. Surprisingly, the symmetries may leave unchanged
a large subspace.
\begin{theorem}\label{th:LK}
Let $K$ be an  origin-symmetric convex body in $\R^n$.
Assume that  there exists isometries $U_1,\ldots,U_m\in \mathcal O_n(K)$ and   coefficients
$c_1,\ldots,c_m>0$ such that, setting $E_i=\fix(U_i)^\perp$,
\begin{equation*}%\label{eq:decomp}
\sum_{i=1}^m c_i P_{E_i} = P_E,
\end{equation*}
for some subspace $E\in \R^n$. 
If the codimension $d$ of $E$ verifies $d\le \alpha n/\log n$, then
$$ L_C\le C(\alpha) \max_i L\big({\rm dim}(E_i)\big).$$
%where for all $i$, $d_i$ is the codimension of $\mathrm{Fix}(U_i)$.
\end{theorem}
\begin{remark}
%Equation~\eqref{eq:decomp} implies that for all $x\in\R^n$, $|P_Ex|^2=\sum c_i |P_{\mathrm{Fix}(U_i)^\bot}x|^2$.
%From this, one deduces 
Recall that $\cap_i \mathrm{Fix}(U_i)=E^\bot$. Hence, the condition on $d$ means 
that $\mathrm{dim}(\cap_i \mathrm{Fix}(U_i))\le  \alpha n/\log n$. In other words, the group of isometries of $K$  may
not act on a subspace of dimension $n/\log(n)$. 
\end{remark}

\begin{proof}
First let us note that we may assume that $K$ is isotropic. Indeed let $A$ be the positive matrix such 
that for all $\theta\in\R^n$, $\int_K (x\cdot\theta)^2 dx= A\theta\cdot\theta$.
Then it is plain that $A^{-1/2}K$ is isotropic. For any isometry $U$ preserving $K$,
$$ A\theta\cdot\theta=  \int_K ( x\cdot\theta)^2 dx=\int_{UK} (x\cdot\theta)^2 dx
= AU^{-1}\theta\cdot U^{-1}\theta.$$
Hence $UAU^*=A$, that is $AU=UA$. Consequently $UA^{-1/2}K=A^{-1/2}UK=A^{-1/2}K$. So $A^{-1/2}K$ is isotropic
and has the same isometric invariances.

It is convenient to set $F_0=E$ and for $i\ge 1$, $F_i=\mathrm{Fix(U_i)}$. Since $P_{E_i}=\mathrm{Id}-P_{F_i}$ for $i=1,\ldots, m$,
the decomposition of the Theorem gives
\begin{equation}\label{eq:decomp2}
 \sum_{i=0}^m c'_i P_{F_i}=\mathrm{Id},
\end{equation}
where $c'_0=(\sum_{j=1}^m c_j)^{-1}$ and for $i\ge 1$, $c'_i=c_j/(\sum_{j=1}^m c_j)$. This decomposition of
the identity allows us to apply the geometric version of the Brascamp-Lieb inequality (see e.g. \cite{B}):
since $K\subset \bigcap_{i=0}^m \{x\in \R^n;\; P_{F_i}x\in P_{F_i}K\}$, 
\begin{eqnarray*}
|K| &\le & \int_{\R^n} \prod_{i=0}^m \mathbf 1_{P_{F_i}K}(P_{F_i}x)^{c_i'} dx \le   \prod_{i=0}^m \left(\int_{F_i} \mathbf 1_{P_{F_i}K}(x_i) \, dx_i\right)^{c'_i}\\
  &=& |P_E K|^{c'_0} \prod_{i=1}^m |P_{F_i}K|^{c'_i}.
\end{eqnarray*}
Set $d_i=\mathrm{dim}(E_i)=n-\mathrm{dim}(F_i)$.
For $i\ge 1$, since $F_i=\mathrm{Fix(U_i)}$ and $U_i$ leaves $K$ invariant, we know by Lemma~\ref{lem:fix} and
Theorem~\ref{th:sectionB} that
$$  |P_{F_i}K|=|K\cap F_i|_{n-d_i}\le |K|^{\frac{n-d_i}{n}} \left(\frac{\kappa L(d_i)}{L_K}\right)^{d_i}.$$
For the projection onto $E$, we first use the Rogers-Shephard inequality
$$ |P_E K|_{n-d} \le  {n \choose d}  \frac{|K|}{|K\cap E^\bot|_{d}}.$$
Next, by a result of  Kannan-Lovasz-Simonovits \cite{KLS},  $ |K|^{1/n}L_K B_2^n \subset K$ (actually for symmetric convex sets this can be found in a stronger
form in the article by Milman and Pajor~\cite{MP}). Taking sections yields
 $|K\cap E^\bot|_d \ge |K|^{d/n} L_K^d |B_2^d|_d$, hence
  $$ |P_EK|_{n-d} \le  {n \choose d}  \frac{|K|^{\frac{n-d}{n}}}{L_K^d |B_2^d|_d }.$$
 Combining these bounds yields
 $$ |K| \le \left({n \choose d}  \frac{|K|^{\frac{n-d}{n}}}{L_K^d |B_2^d|_d }\right)^{c'_0}\prod_{i=1}^m
   \left( |K|^{\frac{n-d_i}{n}} \left(\frac{\kappa L(d_i)}{L_K}\right)^{d_i}\right)^{c'_i}.$$ 
Taking traces in \eqref{eq:decomp2} gives $n=c'_0(n-d)+\sum_{i\ge 1} c'_i(n-d_i)$ so that the terms in $|K|$
cancel out. The latter equality can be also stated as $n-d=\sum_{i\ge 1} c_id_i$. Using also that  $c'=(c'_0,c'_1,\ldots,c'_m)$ is proportional to $c=(1,c_1,\ldots,c_m)$
%, and that taking traces in \eqref{eq:decomp} yields . 
and  rearranging the terms gives 
\begin{eqnarray*}
 L_K &\le&   \left( \frac{{n \choose d}}{|B_2^d|_d}\right)^{\frac{c'_0}{c'_0d+\sum_{j\ge 1} c'_jd_j}}\prod_{i=1}^m
    \big(\kappa L(d_i) \big)^{\frac{c'_id_i}{c'_0d+\sum_{j\ge 1} c'_jd_j}} \\
    &=&  \left( \frac{{n \choose d}^{\frac1d}}{|B_2^d|_d^{\frac1d}}\right)^{\frac{ d}{ n}}\prod_{i=1}^m
    \big(\kappa L(d_i) \big)^{\frac{c_id_i}{ n}} \\
    &\le & (\beta n\sqrt{d})^{\frac{d}{n}} \big(\kappa\max_i L(d_i)\big)^{1-\frac{d}{n}}\\
    &\le & \beta' e^{\frac{3d}{2n} \log n} \max_i L(d_i),
\end{eqnarray*}
where $\beta,\beta'>0$ are universal constants. We have also used that  $\inf_k L(k)>0$.
\end{proof}

\begin{corollary}
   Let $K\subset \R$ be a convex body with barycenter at the origin. Assume that there exists (non necessarily orthogonal)  
  symmetries $S_1,\ldots,S_m$ with respect hyperplanes $H_1,\ldots,H_m$ 
such that for all $i\le m$, $S_{H_i}K =K$ 
and $\mathrm{dim}(\bigcap_{i\le m} H_i)\le \alpha \frac{n}{\log n}$. Then $L_K\le C(\alpha).$
\end{corollary}
\begin{proof}
%Assume as we may that the barycenter of $K$ is at the origin.
  Since any compact subgroup of the linear group is affinely conjugated to a 
subgroup of the orthogonal group, and since the isotropy constant is an affine invariant,
we may assume that the $S_i$ are orthogonal hyperplane symmetries. The reflection group $G$  that
they generate satisfies $\mathrm{Fix}(G) \subset \bigcap_{i\le m} H_i$. Hence, by Lemma~\ref{lem:dec-reflex}, there exists unit vectors $v_1,\ldots,
v_\ell$ and coefficients $c_1,\ldots, c_\ell$ such that $S_{v_i^\bot}K=K$ and 
$$ \sum_{i\le \ell} c_i P_{\R v_i}=P_{\mathrm{Fix}(G)^\perp}.$$
Hence we may apply Theorem~\ref{th:LK}.
\end{proof}

%%%%%%%%%%%%%%%%%%%%%%%%%%%%%%%%%%%%%%%%%%%%%%%%%%%%%%%%%%%

%%%%%%%%%%%%%%%%%%%%%%%%%%%%%%%%%%%%%%%%%%%%%%%%%%%%%%%%%%%

\appendix
\section*{Appendix: Some observations concerning reflections}

A reflection group is a  subgroup of some $\mathcal O_n$ generated by reflections (i.e. by hyperplane symmetries). 
%In our setting, such reflections leave invariant some measure 
%$\mu$ on $\R^n$ or some convex body $K\subset \R^n$.
% , and groups arise as $\mathcal R(\mu)$ or $\mathcal R (K)$, the group generated by the reflections leaving invariant some measure $\mu$ on $\R^n$ or some convex body $K\subset \R^n$.

\begin{lemma}\label{lem:reflection}
   Let $G$ be a finite irreducible reflection group acting on $\R^n$.
 Let $V=\{v\in S^{n-1};\; S_{v^\bot} \in G\}$. Then, using the tensor notation 
 $(v\otimes v)(x)=\langle x,v\rangle v$
for projections on lines,  we have
 $$ \sum_{v\in V} v\otimes v= \frac{\mathrm{card}(V)}{n} \mathrm{Id}.$$
\end{lemma}
\begin{proof}
   Note that for $v\in V$ and $R\in G$, it holds $RS_{v^\bot}R^{-1}=S_{(Rv)^\bot}\in G$. Hence $R$
restricted to $V$ is a permutation of $V$. Set $L=\sum_{v\in V}v\otimes v$. The previous  observation
implies  that
$$ RLR^{-1}= RLR^* =\sum_{v\in V} Rv\otimes Rv=L.$$
Hence $L$ is in the center of $G$. Let $\lambda$ be an eigenvalue of the self-adjoint map $L$ and 
$E$ the corresponding eigenspace. By the above commutation, $E$ is globally invariant by all elements
of $G$ and it is not empty. By irreducibility $E=\R^n$ and $L=\lambda \mathrm{Id}$.
\end{proof}

\begin{lemma}\label{lem:dec-reflex}
 Let $\mathcal R\subset \mathcal O_n$ be a closed reflection group. Then there exist $m\in \mathbb N$, 
 unit vectors $v_1,\ldots,v_m$ such that $S_{v_i^\bot}\in \mathcal R$ for all $i\le m$ 
and coefficients $c_1,\ldots,c_m>0$ such that 
$$ \sum_{i=1}^m c_i v_i\otimes v_i= P_{E},$$
where $E=\mathrm{Fix}(\mathcal R)^\bot$.
\end{lemma}

\begin{proof}
  Classically $\R^n$ can we written as an orthogonal sum of $\mathrm{Fix}(\mathcal R)$ and
of  spaces $E_1,\ldots,E_\ell$ irreducible for the action of $\mathcal R$. Also $\mathcal R$
can be written as a direct product of reflection groups acting, in an irreducible manner, 
 on the $E_i's$. Applying the previous
lemma gives a decomposition of $\mathrm{Id}_{E_i}$, summing them up yields the claimed decomposition
of the identity on $\mathrm{Fix}(\mathcal R)^\bot.$ Actually this argument works when the reflection 
group is finite. However if it is infinite, it can be checked that the group acts on some $E_i$ as 
the whole orthogonal group, and it is not hard to find a decomposition of the identity on $E_i$ since
all unit vectors are allowed.
\end{proof}

In view of the previous lemma, a natural and convenient 
invariance hypothesis to work with is the following: there exists isometries $(U_i)_{i=1}^m$ such that
 $U_iK=K$ and
positive coefficients $(c_i)_{i=1}^m$ such that
\begin{equation}\label{eq:decomp}
\sum_{i=1}^m c_i P_{\mathrm{Fix}(U_i)^\bot} = P_E.
\end{equation} 
This implies that for all $x\in\R^n$, $|P_Ex|^2=\sum c_i |P_{\mathrm{Fix}(U_i)^\bot}x|^2$.
Hence  $\cap_i \mathrm{Fix}(U_i)=E^\bot$. Usually $E$ will be a large space, meaning that the
isometries actually operate on a large part of the space.
Let us provide concrete exemples of invariance hypotheses.

\smallskip
 Unconditional convex bodies have 
attracted a lot of attention. They are invariant by changes of signs of coordinates, or equivalently
by reflection with respect to the coordinate hyperplanes of an orthonormal base $(e_1,\ldots,e_n)$.
In this particular case, \eqref{eq:decomp} boils down to $\sum e_i\otimes e_i=\mathrm{Id}$.

A natural extension is to consider sets $K$ in $\R^{kd}$ which are unconditional by blocks (of size $d$):
$(x_1,\ldots,x_k)\in K \Longrightarrow (\pm x_1,\ldots,\pm x_k)\in K$.
The isometries defined by $R_i:(x_1,\ldots,x_k)\mapsto (x_1,\ldots,x_{i-1},-x_i,x_{i+1},\ldots,x_k)$ satisfy that
$\mathrm{Fix}(R_i)=\{x=(x_1,\ldots,x_k)\in \R^{kd};\; x_i=0\}$ and it is plain that 
$\sum_i P_{\mathrm{Fix}(R_i)^\bot} =\mathrm{Id}.$

This pattern naturally occurs
when considering matricial norms which only depend on the absolute values of matrices.
Let us consider a norm on $M_n(\R)$ of the form $\| A\|=f(A^*A)$ (Schatten norms 
$\|A\|_p=\big(\mathrm{Tr}( (A^*A)^{p/2})\big)^{1/p}$, $p\ge 1$ are the simplest examples).
Then if $$E_i:=\mathrm{Diag}(1,\ldots,1,\underbrace{-1}_{i},1,\ldots,1),$$
 then the maps
$R_i:A\mapsto E_iA$ are isometries of $(M_n(\R),\|\cdot\|_2)$ since $(E_iA)^*E_iA=A^*A$.
Note that $R_iA$ is obtained from $A$ by changing the signs of all the entries of the $i$th 
row of $A$. Consequently the unit ball of any  Schatten norm is unconditional by blocks of size
$n$. Note that 
$\mathrm{Fix}(R_i)=\{A\in M_n(\R); a_{i,j}=0,\, \forall j\le n\}$ 
has codimension $n$  while  the ambient space is of dimension $n^2$. Plainly $\sum_i P_{\mathrm{Fix}(R_i)^ \bot}=\mathrm{Id}$.

\medskip
Thanks to Coxeter's classification of irreducible finite reflection groups, one may 
obtain many concrete examples of invariances. Among them, let us emphasize
the group of isometries of a regular simplex $\Delta_n\subset \R^n$ denoted $\mathcal O(\Delta_n)$.
By restriction to the vertices $\{u_1,\ldots,u_{n+1}\}$ of $\Delta_n$ it is identified
with the set of permutation of these vertices. The group $\mathcal O(\Delta_n)$ contains exactly 
$n(n+1)/2$ reflections: namely the $S_{(u_i-u_j)^\bot}$, for $1\le i<j\le n+1$. They correspond to 
transpositions.

Another natural invariance hypothesis is related to the regular simplex: namely the exchangeability
condition. A set or a measure is exchangeable if it is invariant by permutations of coordinates.
Here one considers the isometries of $\R^n$  given by 
 $$ R_{\sigma}:(x_1,\ldots,x_n)\mapsto (x_{\sigma(1)}, \ldots, x_{\sigma(n)}),$$
where $\sigma\in \mathcal S_n$. The group $G=\{R_\sigma;\; \sigma\in \mathcal S_n\}$ has a line of fixed
points: $\mathrm{Fix}(G)=\R v_n$, where $v_n=(1,\ldots,1)$. It is a reflection group generated by the images of
transpositions $R_{\tau_{i,j}}$. On $v_n^\bot$, $G$ acts as $\mathcal O(\Delta_{n-1})$.
Indeed, it permutes the points $(e_i-v_n/n)_{i=1}^n$ which form a regular simplex of $v_n^\bot$.
 Lemma~\ref{lem:dec-reflex} thus provides a decomposition of the form \eqref{eq:decomp}
  $$\frac2n \sum_{i<j} \frac{e_i-e_j}{2}\otimes\frac{e_i-e_j}{2}=P_{v_n^ \bot}.$$ 
One could in the same way introduce a block-exchangeability condition and derive a corresponding decomposition.

\medskip
We can also consider invariances involving only direct isometries. For instance, the next statement encompasses the set $SO(\Delta_n)$
of direct isometries of the simplex.
\begin{lemma}\label{lem:dec-direct}
Let $G$ be a finite reflection group  acting on $\R^n$, $n\ge 2$. Set $E=\mathrm{Fix}(G)^\bot$. Then there exists $m$, rotations $U_1,\ldots, U_m\in G\cap SO(n)$ and coefficients $c_1,\ldots,c_m\ge 0$ 
such that
$$
\sum_{i=1}^m c_i P_{\mathrm{Fix}(U_i)^\bot} = P_F,
$$
where $F$ is $E$ or a hyperplane of $E$ (the latter occurs when $G$ has an odd number of 
one-dimensional irreducible factors). Note that by definition $\mathrm{dim}(\mathrm{Fix}(U_i)^\bot)=2$.
\end{lemma}
\begin{proof}
First assume that $G$ is irreducible and acts on $\R^n$ with $n\ge 2$.
Since it is generated by reflections, the direct isometries in $G$ are generated by products of two
reflections, that is rotations (since $n\ge 2$ there are at least two distinct reflections).
Let $R\in G\cap SO(n)$ be such a rotation and  consider the plane $\Pi=\mathrm{Fix}(R)^\bot$.
For all $U\in G$, $URU^{-1}\in G\cap SO(n)$ and $\mathrm{Fix}(URU^{-1})=U \mathrm{Fix}(R)^\bot=U\Pi$.
Also note that for every subspace $\Sigma$, $UP_\Sigma U^{-1}=P_{U\Sigma}$.
Consider $$L=\sum_{U\in G} P_{\mathrm{Fix}(URU^{-1})^\bot}=\sum_{U\in G} P_{U\Pi}.$$
From the above remarks and the group property, for all $V\in G$, $VLV^{-1}=L$. So $L$ commutes
with all the elements of $G$. Since $L$ is a symmetric positive map, it has at least a non-zero eigenvalue
$\lambda$. By the commutation, the eigenspace $E_\lambda$ is stable by $G$ and thus by irreducibility
it is the whole space. Hence $L=\lambda \mathrm{Id}$. This proves the claim for an irreducible group $G$.
 
 For a general group, we consider the induced irreducible decomposition. On components of dimension at least
 $2$ we apply the above argument. We group the one-dimensional components by two. On each such plane the 
 the rotation of angle $\pi$ is in the group, as the product of minus identity on each irreducible line.
 The decomposition of the identity of this plane is obvious (the origin is the only fixed point of the rotation).
 Summing up all these decompositions, we obtain the claim. Note that when there is an odd number of one dimensional
 irreducible components, one of them is left aside. 
\end{proof}

Finally, we recall a useful lemma, which uses more of the explicit description of finite reflection groups, see \cite{BF}.
\begin{lemma}\label{lem:coxeter}
Let $G$ be a reflection group on $\R^n$. Assume that the set of its reflections is closed. If $\mathrm {Fix}(G)=\{0\}$
then there exists an orthogonal decomposition $\R^n=F_1\bigoplus \cdots \bigoplus F_\ell$ and
polytopes $P_i\subset F_i$ for all $i\le \ell$ such that 
 $P_i$ is a (full dimension) regular simplex in $F_i$, or a regular polygon if $F_i$ is of dimension 2, such 
 that 
 $$ \mathcal O(P_1)\times \cdots\times \mathcal O(P_\ell)\subset G.$$ 
\end{lemma}
 If $E_i$
is a line, then by convention $P_i$ is  a symmetric interval. Note that $\mathrm{Fix}(\mathcal O(P_1)\times \cdots\times \mathcal O(P_\ell))=\{0\}$. The interest of this result is to provide a simple reflection subgroup
of $G$ with no nontrivial fixed points.

\bigskip

\noindent {\bf Acknowledgments:} We would like to thank Pietro Caputo, Matthieu Fradelizi, Michael Loss, Felix Otto and Alain Pajor for useful discussions. We also gratefully
acknowledge the hospitality of the Newton Institute of Mathematical Sciences, Cambridge, where part of this work was done.

%%%%% BIBLIO  %%%%%%

%\bibliographystyle{amsplain}
%\bibliography{bibli}

%%%%%%%%%%%%%%%%
\vskip1cm

\noindent
F. Barthe: Institut de Math\'ematiques de Toulouse,  Universit\'e Paul Sabatier, 31062 Toulouse cedex 09, France.

\noindent
Email: barthe@math.univ-toulouse.fr

\bigskip

\noindent
D. Cordero-Erausquin: Institut de Math\'ematiques de Jussieu, Universit\'e Pierre et Marie Curie (Paris 6),
75252 Paris Cedex 05, France
 
\noindent
Email: cordero@math.jussieu. fr

\end{document}